\RequirePackage{fix-cm}
\documentclass[smallextended]{svjour3}       % onecolumn (second format)
\smartqed  % flush right qed marks, e.g. at end of proof
\usepackage{amsmath}
\usepackage{graphicx}
\usepackage{subcaption}
\usepackage{lipsum}
\usepackage{amsfonts}
\usepackage{epstopdf}
\usepackage{algorithm,algorithmic}
\usepackage{tikz}
\usetikzlibrary{arrows.meta}

\usepackage{booktabs}

\ifpdf
\DeclareGraphicsExtensions{.eps,.pdf,.png,.jpg}
\else
\DeclareGraphicsExtensions{.eps}
\fi

\newcommand{\R}{\mathbb{R}}
\newcommand{\Hq}{\mathbb{H}}

\begin{document}

\title{Order-Specified Quaternion CP Decomposition: Models, Algorithms, and Applications\thanks{This work was funded in part by the National Natural Science
		Foundation of China under Grant Nos. 12501714,
		the Yunnan Fundamental Research Projects under Grant Nos.
		202401AU070203 and 202501AT070192, the Anhui Postdoctoral
		Scientific Research Program Foundation of the Anhui Provincial
		Department of Human Resources and Social Security under
		Grant No. 2025C1139, and the Major Project of the Natural
		Science Research Project of Anhui Provincial Universities
		under Grant No. 2025AHGXZK20062.}
}

\titlerunning{Order-Specified Quaternion CP Decomposition}        % if too long for running head

\author{Jifei Miao \and
        Niansheng Tang \and Michael K. Ng \and Juan Han \and  Kit Ian Kou
}


\institute{Jifei Miao \at
School of Mathematics and Statistics, Yunnan University, Kunming, Yunnan, 650091, China
\email{jifmiao@163.com}          
\and
Niansheng Tang \at 
Yunnan Key Laboratory of Statistical Modeling and Data Analysis, Yunnan University, Kunming 650500, China
\email{nstang@ynu.edu.cn}  
\and
Michael K. Ng\at Department of Mathematics, Hong Kong Baptist University, Kowloon Tong, Hong Kong, China
\email{michael-ng@hkbu.edu.hk}  
\and
Juan Han \at School of Mathematics and Physics, Anhui Jianzhu University, Hefei, Anhui, 230601, China
\email{juanhan0604@163.com}
\and 
Kit Ian Kou \at Department of Mathematics, Faculty of Science, University of Macau
\email{kikou@umac.mo}
}


\maketitle

\begin{abstract}
Tensor CP decomposition serves as a fundamental tool for representing multiway data. While quaternion tensors are well suited to modeling color visual data, extending CP decomposition to the quaternion domain is nontrivial due to the noncommutativity of quaternion multiplication. This paper develops an order-specified quaternion CP decomposition (OS-QCP) for third-order tensors, allowing all three factor matrices to be quaternion-valued and explicitly incorporating their multiplication order into the model. We show that different multiplication orders can yield different OS-QCP ranks for the same quaternion tensor, characterize the permutation and quaternion scaling indeterminacies, and establish a sufficient condition for essential uniqueness for any prescribed multiplication order. An explicit example further shows that OS-QCP rank can depend on the multiplication order even when the complex representation of the quaternion tensor admits an essentially unique CP decomposition. An alternating least-squares (ALS) algorithm is derived by expressing quaternion left and right multiplication as real linear operators, reducing each factor update to real least-squares problems. To accommodate order dependence in tensor completion, we further introduce an order-ensemble quaternion CP (OE-QCP) model that incorporates different multiplication orders through nonnegative weights. The model jointly estimates three quaternion factor matrices and the order weights, with each single-order OS-QCP model included as a special case. The resulting completion problem is solved by a parallel row-wise masked ALS algorithm. Experiments on color video inpainting demonstrate the effectiveness and competitive performance of the proposed method.
%\subclass{MSC code1 \and MSC code2 \and more}
\end{abstract}

\section{Introduction}
\label{sec:introduction}

Tensor decompositions provide compact and interpretable re-presentations for multiway data and have become fundamental tools in numerical multilinear algebra, signal processing, machine learning, and imaging sciences \cite{hong2020generalized,sidiropoulos2017tensor,cichocki2015tensor,LiZZC25}. Among various tensor decomposition models, the CANDECOMP/PARAFAC (CP) decomposition is one of the most classical and widely used formats \cite{harshman1970foundations}. It represents a high-dimensional tensor as a sum of rank-one components and therefore provides a simple yet powerful way to describe separable structures along different modes \cite{kolda2009tensor}. Due to its low parameter complexity and clear multilinear interpretation, CP decomposition has been successfully applied to data compression, feature extraction, blind source separation, recommender systems, hyperspectral image analysis, and video processing \cite{chen2022unsupervised,chabriel2014joint,zhao2022fast,liu2019low}. For a third-order real tensor
\(\mathcal X\in\mathbb R^{I_1\times I_2\times I_3}\),
the classical CP decomposition takes the entry-wise form
\begin{equation}
	x_{ijk}
	=
	\sum_{r=1}^{R}a_{ir}b_{jr}c_{kr},
	\label{eq:intro_real_cp}
\end{equation}
where the smallest possible number of components \(R\)
is called the CP rank and the factor matrices characterize latent structures along the three tensor modes. 
Since the factor entries in \eqref{eq:intro_real_cp} are real-valued
and real scalar multiplication is commutative, all six multiplication
orders yield the same rank-one term.

In recent years, quaternion-valued models have attracted increasing attention in color image processing \cite{miao2021color,jia2025new,chen2022color}, color video processing \cite{he2023eigenvalues,jia2022non}, and polarimetric imaging \cite{kawai2025quaternion,Pan2022SeparableQM,flamant2020quaternion}. A quaternion can naturally encode a scalar part together with three coupled imaginary components. In particular, a color pixel with red, green, and blue (RGB) components can be represented as a pure quaternion $
\dot{q}
=
R\mathbf{i}
+
G\mathbf{j}
+
B\mathbf{k}$ \cite{miao2021color},
which treats the three color channels as an integrated vector-valued quantity rather than three independent scalar channels. This representation helps preserve the intrinsic coupling among color channels and has motivated many quaternion matrix models, such as quaternion matrix factorization \cite{miao2020quaternion} and quaternion low-rank approximation \cite{chen2019low,song2021low}. Beyond matrices, quaternion tensor models have also shown strong potential in color video and multidimensional visual data processing \cite{miao2020low,yang2026nonlinear}, because they can simultaneously preserve the multiway structure of the data and the internal coupling among vector-valued channels. Therefore, integrating classical tensor decompositions with quaternion representations is expected to provide new modeling tools for color visual data processing. As CP decomposition is one of the most fundamental tensor decomposition models, this paper investigates its quaternion extension with all factor matrices allowed to be quaternion-valued.

However, allowing all factor matrices in a quaternion CP decomposition
to be quaternion-valued is nontrivial because quaternion multiplication
is noncommutative. The multiplication order of the factor entries
must therefore be explicitly specified as part of the model.
In particular, changing the multiplication order can alter the
interaction between the vector parts of the quaternion factors
and hence the resulting product. To the best of our knowledge, the quaternion CP model developed
by Flamant et al.~\cite{flamant2026multilinear} is the only
existing CP formulation for quaternion tensors.
Their model is established within a framework of
\(\mathbb{HR}\)-multilinear forms.
For a third-order quaternion tensor, it takes the form
\begin{equation}
	\dot{x}_{ijk}
	=
	\sum_{r=1}^{R}
	\dot{a}_{ir} b_{jr} \dot{c}_{kr},
	\label{eq:intro_existing_qcp}
\end{equation}
where \(\dot{a}_{ir},\dot{c}_{kr}\in\Hq\) and
\(b_{jr}\in\R\).
Thus, the first and third factor matrices are quaternion-valued,
whereas the middle factor matrix is real-valued.
For an \(N\)-th order tensor, the formulation similarly permits
quaternion-valued factors in the first and last modes, with
real-valued factors in the remaining \(N-2\) modes.
To distinguish this formulation from the model considered here,
we refer to it as partial quaternion CP (PQCP).

The PQCP framework provides a principled approach to quaternion
tensor decomposition, but its real-valued intermediate factors
restrict the admissible factorization structure.
In particular, the middle factor in
\eqref{eq:intro_existing_qcp} contributes a real scalar to
each component, so quaternion multiplication involves only
the two outer factor entries.
Allowing all three factor matrices to be quaternion-valued
removes this restriction and permits quaternion-valued
interactions among factors associated with all three modes.
This raises the question of how to define, analyze, and compute
a CP decomposition in which the multiplication order of these
three quaternion entries is explicitly incorporated.

Motivated by this observation, this paper develops a full quaternion-valued CP framework in which all factor matrices are quaternion-valued. The proposed model can be extended to quaternion tensors of arbitrary
order. This paper focuses on third-order quaternion tensors, which
naturally represent color videos by encoding the RGB channels as
quaternion entries and retaining the two spatial modes and the
temporal mode.
Since each entry of a third-order quaternion rank-one term involves
three quaternion factors, six multiplication orders are possible.
A central question is whether changing the order leads to genuinely
different decompositions even when the factor matrices are allowed
to vary freely. To address this question, we introduce an
order-specified quaternion CP decomposition (OS-QCP) and show that
different orders can require different ranks for the same tensor.
We characterize its permutation and quaternion scaling indeterminacies
and establish a sufficient condition for essential uniqueness for
any prescribed multiplication order. An explicit example further
shows that essentially unique decompositions under different orders
can have equal ranks without their factors being equivalent under
these transformations. Moreover, order dependence can persist even
if the complex representation of the quaternion tensor admits
an essentially unique complex CP decomposition.

We also establish the existence of best rank-one approximations
and regularized low-rank solutions, and derive an alternating
least-squares (ALS) algorithm using the real-linear operators
induced by left and right quaternion multiplication.
Each factor update reduces to a real least-squares problem
while preserving the prescribed multiplication order.

In practical applications, there is often no clear a priori reason
to favor one multiplication order. Building on OS-QCP, we therefore
introduce an order-ensemble quaternion CP (OE-QCP) model that
incorporates all six orders and use it for quaternion tensor completion.
For each component, the same three quaternion factor entries are
multiplied in six orders, and the resulting products are combined
using nonnegative weights. Summing these weighted components gives
the represented tensor, as illustrated in
Fig.~\ref{fig:oeqcp_framework}.

OE-QCP jointly estimates three quaternion factor matrices and the
order weights, allowing the contribution of each order to adapt to
the observed data. All six orders act on the same factors, so the
model does not require six independent sets of factor matrices.
When one order weight equals one and the others vanish, OE-QCP
reduces to the corresponding OS-QCP model. If all factor entries
are real-valued or belong to a common complex subfield, the ordered
products coincide, and OE-QCP reduces to the corresponding real
or complex CP representation. The resulting tensor completion
problem is solved by a parallel row-wise masked ALS algorithm.

The main contributions of this paper are summarized as follows.
\begin{itemize}
	\item We introduce an order-specified quaternion CP decomposition
	(OS-QCP) with all factor matrices quaternion-valued.
	We show that multiplication order can affect both the rank and
	the factor matrices beyond permutation and quaternion scaling
	indeterminacies, and establish a sufficient condition for
	essential uniqueness for any prescribed multiplication order.
	An explicit example further demonstrates that an essentially
	unique complex CP decomposition of the complex representation
	does not eliminate this order dependence.
	
	\item We establish the existence of best rank-one approximations
	and regularized low-rank solutions for OS-QCP, and develop an
	alternating least-squares (ALS) algorithm.
	By exploiting the real-linear operators induced by left and
	right quaternion multiplication, each factor update reduces
	to a real least-squares problem while preserving the
	prescribed multiplication order.
	
	\item We propose an order-ensemble quaternion CP (OE-QCP) model
	that jointly estimates quaternion factor matrices and
	nonnegative weights for different multiplication orders,
	with every single-order OS-QCP model included as a special case.
	For tensor completion, we develop a parallel row-wise masked
	ALS algorithm. Experiments on color video inpainting
	demonstrate the effectiveness and competitive performance
	of the proposed method.
\end{itemize}

The remainder of this paper is organized as follows.
Section~\ref{sec:notations_preliminaries} introduces the notation
and preliminaries.
Section~\ref{sec:osqcp} defines OS-QCP and analyzes its order
dependence, essential uniqueness, and approximation properties.
Section~\ref{sec:als} derives the ALS algorithm for a prescribed
multiplication order.
Section~\ref{sec:completion} develops the OE-QCP model and its
associated algorithm for quaternion tensor completion.
Section~\ref{sec: experiments} reports numerical results on order
sensitivity and color video inpainting, together with ablation studies.
Section~\ref{sec:conclusion} concludes the paper.

\section{Notation and Preliminaries}
\label{sec:notations_preliminaries}

We first summarize the main notation used throughout the paper in Table~\ref{tab:notations}. 
\begin{table}[h]
	\centering
	\caption{Main notation used in the paper.}
	\label{tab:notations}
	\renewcommand{\arraystretch}{1.05}
	\begin{tabular}{p{0.27\linewidth}p{0.65\linewidth}}
		\toprule
		\textbf{Notation} & \textbf{Description} \\
		\midrule
		
		\(\mathbb R,\mathbb C,\Hq\)
		& real field, complex field, and quaternion skew field \\
		
		\(\dot q\)
		& quaternion scalar
		\(q_0+q_1\mathbf i+q_2\mathbf j+q_3\mathbf k\),
		\(q_l\in\mathbb R\) \\
		
		\(\overline{\dot q},|\dot q|\)
		& conjugate and modulus of \(\dot q\) \\
		
		\(\dot{\mathbf X},\dot{\mathcal X}\)
		& quaternion-valued matrix and tensor \\
		
		\((\cdot)^T,(\cdot)^H\)
		& transpose and conjugate transpose \\
		
		\(\|\cdot\|_F\)
		& Frobenius norm \\
		
		\(\operatorname{vec}_{\mathbb R}(\dot q)\)
		& real coordinate vector
		\([q_0,q_1,q_2,q_3]^T\) \\
		
		\(\mathbf L(\dot q),\mathbf R(\dot q)\)
		& left/right multiplication matrices of \(\dot q\) \\
		
		\(\pi,S_3\)
		& multiplication order and permutations of
		\(\{1,2,3\}\) \\
		
		\(\Pi_\pi\)
		& ordered product specified by \(\pi\) \\
		
		\(R_\pi\)
		& OS-QCP rank under order \(\pi\) \\
		
		\(\omega_\pi,\boldsymbol\omega\)
		& order weight and weight vector \\
		
		\(\Omega,\mathcal P_\Omega\)
		& observed index set and projection operator \\
		
		\bottomrule
	\end{tabular}
\end{table}

The quaternion algebra \(\Hq\) is a four-dimensional associative algebra over \(\R\) with basis \(\{1,\mathbf{i},\mathbf{j},\mathbf{k}\}\) \cite{hamilton1866elements}. A quaternion \(\dot{q}\in\Hq\) is written as
\begin{equation}
	\dot{q}
	=
	q_0+q_1\mathbf{i}+q_2\mathbf{j}+q_3\mathbf{k},
	\qquad
	q_0,q_1,q_2,q_3\in\R.
\end{equation}
The imaginary units satisfy
\begin{equation}
	\begin{cases}
		\mathbf{i}^2=\mathbf{j}^2=\mathbf{k}^2=\mathbf{i}\mathbf{j}\mathbf{k}=-1,\\
		\mathbf{i}\mathbf{j}=-\mathbf{j}\mathbf{i}=\mathbf{k},\quad
		\mathbf{j}\mathbf{k}=-\mathbf{k}\mathbf{j}=\mathbf{i},\quad
		\mathbf{k}\mathbf{i}=-\mathbf{i}\mathbf{k}=\mathbf{j}.
	\end{cases}
	\label{imagm}
\end{equation}
Thus, it is clear from \eqref{imagm} that quaternion multiplication is noncommutative.

The conjugate and modulus of \(\dot{q}\) are defined by
\begin{equation*}
	\overline{\dot{q}}
	=
	q_0-q_1\mathbf{i}-q_2\mathbf{j}-q_3\mathbf{k}\ \text{and}\  |\dot{q}|
	=
	\sqrt{\dot{q}\overline{\dot{q}}}
	=
	\sqrt{q_0^2+q_1^2+q_2^2+q_3^2}.
\end{equation*}
If \(\dot{q}\neq0\), then its inverse is given by $\dot{q}^{-1} = \frac{\overline{\dot{q}}}{|\dot{q}|^2}$, which satisfies $\dot{q}^{-1}\dot{q} = \dot{q}\dot{q}^{-1} = 1$. 

For the subsequent algorithmic derivations, a quaternion $\dot q=q_0+q_1\mathbf i+q_2\mathbf j+q_3\mathbf k$
is identified with its real coordinate vector
\begin{equation}
	\operatorname{vec}_{\R}(\dot q)
	=
	\begin{bmatrix}
		q_0&q_1&q_2&q_3
	\end{bmatrix}^{T}
	\in\R^4.
	\label{eq:real_coordinate_vector}
\end{equation}
Following the standard real matrix forms of quaternion left and right
multiplication \cite{valverde2018dual}, define
\begin{equation}
	\mathbf L(\dot l)
	=
	\begin{bmatrix}
		l_0 & -l_1 & -l_2 & -l_3\\
		l_1 &  l_0 & -l_3 &  l_2\\
		l_2 &  l_3 &  l_0 & -l_1\\
		l_3 & -l_2 &  l_1 &  l_0
	\end{bmatrix}
	\in\R^{4\times4},
	\label{eq:left_matrix}
\end{equation}
for
\(\dot l=l_0+l_1\mathbf i+l_2\mathbf j+l_3\mathbf k\), and
\begin{equation}
	\mathbf R(\dot r)
	=
	\begin{bmatrix}
		r_0 & -r_1 & -r_2 & -r_3\\
		r_1 &  r_0 &  r_3 & -r_2\\
		r_2 & -r_3 &  r_0 &  r_1\\
		r_3 &  r_2 & -r_1 &  r_0
	\end{bmatrix}
	\in\R^{4\times4},
	\label{eq:right_matrix}
\end{equation}
for
\(\dot r=r_0+r_1\mathbf i+r_2\mathbf j+r_3\mathbf k\).
These matrices satisfy
\begin{align}
	\operatorname{vec}_{\R}(\dot l\dot q)
	&=
	\mathbf L(\dot l)\operatorname{vec}_{\R}(\dot q),
	\label{leftm}\\
	\operatorname{vec}_{\R}(\dot q\dot r)
	&=
	\mathbf R(\dot r)\operatorname{vec}_{\R}(\dot q).
	\label{rightm}
\end{align}

Combining the two matrix forms gives the following elementary result, which
will be used to formulate the factor-update subproblems as real least-squares
problems.

\begin{lemma}[Real-linearity of two-sided quaternion multiplication]
	\label{lem:real_linearity}
	For fixed \(\dot l,\dot r\in\Hq\), the mapping
	\(\mathcal T_{\dot l,\dot r}:\Hq\to\Hq\),
	defined by
	\(\mathcal T_{\dot l,\dot r}(\dot q)=\dot l\dot q\dot r\),
	is linear over \(\R\). Its real matrix form is
	\begin{equation}
		\operatorname{vec}_{\R}(\dot l\dot q\dot r)
		=
		\mathbf L(\dot l)\mathbf R(\dot r)
		\operatorname{vec}_{\R}(\dot q).
		\label{eq:two_sided_real_representation}
	\end{equation}
\end{lemma}

\begin{proof}
	By \eqref{rightm},
	\(\operatorname{vec}_{\R}(\dot q\dot r)
	=\mathbf R(\dot r)\operatorname{vec}_{\R}(\dot q)\).
	Applying \eqref{leftm} to \(\dot q\dot r\) gives
	\begin{equation*}
		\operatorname{vec}_{\R}(\dot l\dot q\dot r)
		=
		\mathbf L(\dot l)
		\operatorname{vec}_{\R}(\dot q\dot r)
		=
		\mathbf L(\dot l)\mathbf R(\dot r)
		\operatorname{vec}_{\R}(\dot q),
	\end{equation*}
	which proves \eqref{eq:two_sided_real_representation}.
	Since \(\mathbf L(\dot l)\mathbf R(\dot r)\) is a real
	matrix independent of \(\dot q\), the mapping
	\(\mathcal T_{\dot l,\dot r}\) is linear over \(\R\).
\end{proof}

A quaternion tensor is a tensor whose entries belong to the quaternion
skew field \(\Hq\)~\cite{miao2020low}. A third-order quaternion tensor is
denoted by
\[
\dot{\mathcal X}
=
(\dot{x}_{ijk})
\in
\Hq^{I_1\times I_2\times I_3},
\]
where
\(\dot{x}_{ijk}\in\Hq\) for
\(i=1,\ldots,I_1\),
\(j=1,\ldots,I_2\), and
\(k=1,\ldots,I_3\).
The Frobenius norm of a quaternion tensor is defined as
\[
\|\dot{\mathcal X}\|_F^2
=
\sum_{i,j,k}
|\dot{x}_{ijk}|^2 .
\]
Unlike real and complex tensors, quaternion tensors involve
noncommutative multiplication of their entries, which leads to
additional considerations in tensor decomposition models. 

For further details on quaternion algebra, matrices, and tensors,
see \cite{zhang1997quaternions,rodman2014topics,miao2024quaternion,miao2020low}.

\section{Order-Specified Quaternion CP Decomposition}
\label{sec:osqcp}
This section develops OS-QCP for third-order quaternion tensors.
Each component involves three quaternion factor entries, which
admit six multiplication orders. Since their products may depend
on the order, it must be specified as part of the decomposition.

\subsection{OS-QCP Decomposition and Order-Specific Rank}

Let \(S_3\) denote the set of all permutations of \(\{1,2,3\}\). For a
prescribed order \(\pi\in S_3\), let
\(\Pi_{\pi}(\dot a,\dot b,\dot c)\) denote the quaternion product of
\(\dot a\), \(\dot b\), and \(\dot c\) arranged according to \(\pi\). For example,
\(\Pi_{213}(\dot a,\dot b,\dot c)=\dot b\dot a\dot c\).

\begin{definition}[Order-specified quaternion outer product]
	\label{defrankone}
	For a prescribed order \(\pi\in S_3\), consider three quaternion
	vectors
	\(\dot{\mathbf a}\in\Hq^{I_1}\),
	\(\dot{\mathbf b}\in\Hq^{I_2}\), and
	\(\dot{\mathbf c}\in\Hq^{I_3}\).
	Their order-specified quaternion outer product is the tensor
	\begin{equation}
		\dot{\mathcal T}^{(\pi)}
		=
		\circ_{\Hq,\pi}
		\left(
		\dot{\mathbf a},
		\dot{\mathbf b},
		\dot{\mathbf c}
		\right)
		\in
		\Hq^{I_1\times I_2\times I_3},
		\label{eq:os_outer_product}
	\end{equation}
	where
	\(\circ_{\Hq,\pi}(\cdot,\cdot,\cdot)\)
	denotes the three-factor quaternion outer product under the
	prescribed multiplication order \(\pi\).
	Specifically, its entries are defined by
	\begin{equation}
		\dot t_{ijk}^{(\pi)}
		=
		\Pi_{\pi}
		\left(
		\dot a_i,
		\dot b_j,
		\dot c_k
		\right).
		\label{eq:os_rank_one_entry}
	\end{equation}
	If nonzero, \(\dot{\mathcal T}^{(\pi)}\) is called an
	order-\(\pi\) quaternion rank-one tensor.
\end{definition}

These order-specified outer products serve as the components of the
following OS-QCP decomposition.

\begin{definition}[Order-specified quaternion CP decomposition]
	\label{def:os_qcp}
	Consider a third-order quaternion tensor
	\(\dot{\mathcal X}\in\Hq^{I_1\times I_2\times I_3}\) and a prescribed
	multiplication order \(\pi\in S_3\). Its order-specific quaternion
	CP rank, denoted by $
	R_\pi
	=
	\operatorname{rank}_{\Hq,\pi}(\dot{\mathcal X})$,
	is the smallest nonnegative integer for which there exist
	quaternion vectors
	\(\dot{\mathbf a}_r^{(\pi)}\in\Hq^{I_1}\),
	\(\dot{\mathbf b}_r^{(\pi)}\in\Hq^{I_2}\), and
	\(\dot{\mathbf c}_r^{(\pi)}\in\Hq^{I_3}\)
	satisfying
	\begin{equation}
		\dot{\mathcal X}
		=
		\sum_{r=1}^{R_\pi}
		\circ_{\Hq,\pi}
		\left(
		\dot{\mathbf a}_{r}^{(\pi)},
		\dot{\mathbf b}_{r}^{(\pi)},
		\dot{\mathbf c}_{r}^{(\pi)}
		\right).
		\label{eq:order_specific_qcp_rank}
	\end{equation}
	A representation attaining this minimum is called an
	order-\(\pi\) OS-QCP decomposition. In entry-wise form,
	\begin{equation}
		\dot{x}_{ijk}
		=
		\sum_{r=1}^{R_\pi}
		\Pi_{\pi}\left(
		\dot a_{ir}^{(\pi)},
		\dot b_{jr}^{(\pi)},
		\dot c_{kr}^{(\pi)}
		\right).
		\label{eq:general_osqcp}
	\end{equation}
	The corresponding factor matrices are $
	\dot{\mathbf A}^{(\pi)}
	=
	[\dot a_{ir}^{(\pi)}]
	\in\Hq^{I_1\times R_\pi},
	\dot{\mathbf B}^{(\pi)}
	=
	[\dot b_{jr}^{(\pi)}]
	\in\Hq^{I_2\times R_\pi}$, and
	$\dot{\mathbf C}^{(\pi)}
	=
	[\dot c_{kr}^{(\pi)}]
	\in\Hq^{I_3\times R_\pi}$,
	whose \(r\)th columns are
	\(\dot{\mathbf a}_{r}^{(\pi)}\),
	\(\dot{\mathbf b}_{r}^{(\pi)}\), and
	\(\dot{\mathbf c}_{r}^{(\pi)}\), respectively.
	The zero tensor has rank zero, with the empty sum interpreted as zero.
\end{definition}

For any prescribed order \(\pi\in S_3\), a simple exact representation
of \(\dot{\mathcal X}\) is obtained by summing the outer products
\(\circ_{\Hq,\pi}
(\mathbf e_i\dot x_{ijk},\mathbf e_j,\mathbf e_k)\)
over its nonzero entries.
Here, \(\mathbf e_i\), \(\mathbf e_j\), and \(\mathbf e_k\)
are the standard real coordinate vectors in dimensions
\(I_1\), \(I_2\), and \(I_3\), respectively.
Each outer product equals \(\dot x_{ijk}\) at position \((i,j,k)\)
and zero elsewhere, independently of \(\pi\).
Thus, the admissible component counts form a nonempty set of
nonnegative integers and have a minimum, so \(R_\pi\) is well defined
and satisfies \(0\leq R_\pi\leq I_1I_2I_3\).

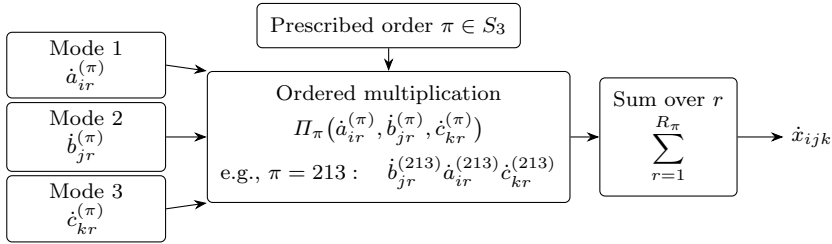
\begin{figure}[h]
	\centering
	\begin{tikzpicture}[
		font=\small,
		>=Stealth,
		factor/.style={
			draw,
			rounded corners=2pt,
			align=center,
			minimum width=2.1cm,
			minimum height=0.75cm,
			inner sep=3pt
		},
		operation/.style={
			draw,
			rounded corners=2pt,
			align=center,
			inner sep=5pt
		}
		]
		\node[factor] (a) at (0,0.95)
		{Mode \(1\)\\[1pt]
			\(\dot a_{ir}^{(\pi)}\)};
		
		\node[factor] (b) at (0,0)
		{Mode \(2\)\\[1pt]
			\(\dot b_{jr}^{(\pi)}\)};
		
		\node[factor] (c) at (0,-0.95)
		{Mode \(3\)\\[1pt]
			\(\dot c_{kr}^{(\pi)}\)};
		
		\node[operation] (product) at (4,0)
		{Ordered multiplication\\[2pt]
			\(\Pi_\pi
			\bigl(
			\dot a_{ir}^{(\pi)},
			\dot b_{jr}^{(\pi)},
			\dot c_{kr}^{(\pi)}
			\bigr)\)\\[3pt]
			e.g., \(\pi=213:\quad
			\dot b_{jr}^{(213)}
			\dot a_{ir}^{(213)}
			\dot c_{kr}^{(213)}\)};
		
		\node[operation, anchor=south] (order)
		at ([yshift=0.3cm]product.north)
		{Prescribed order \(\pi\in S_3\)};
		
		\node[operation] (sum) at (7.7,0)
		{Sum over \(r\)\\[2pt]
			\(\displaystyle\sum_{r=1}^{R_\pi}\)};
		
		\node (output) at (9.6,0)
		{\(\dot x_{ijk}\)};
		
		\draw[->] (a.east) -- (product.north west);
		\draw[->] (b.east) -- (product.west);
		\draw[->] (c.east) -- (product.south west);
		\draw[->] (order.south) -- (product.north);
		\draw[->] (product.east) -- (sum.west);
		\draw[->] (sum.east) -- (output.west);
	\end{tikzpicture}
	\caption{Entrywise structure of the OS-QCP decomposition.}
	\label{fig:osqcp_structure}
\end{figure}

\begin{remark}[Relationships with existing CP models]
	Restricting all factor entries to \(\mathbb R\) or a common
	complex subfield of \(\Hq\) gives the corresponding real
	or complex CP representation, with all six orders yielding
	the same products. The associated rank is defined subject
	to the corresponding factor restriction.
	Under order \(123\), restricting \(\dot{\mathbf B}\) to
	be real-valued gives the PQCP form considered in
	\cite{flamant2026multilinear}.
\end{remark}

As illustrated in Figure~\ref{fig:osqcp_structure}, the prescribed
multiplication order determines the positions of the factor entries
within each product, while
\(\dot{\mathbf A}^{(\pi)}\),
\(\dot{\mathbf B}^{(\pi)}\), and
\(\dot{\mathbf C}^{(\pi)}\)
remain associated with the first, second, and third modes, respectively.
The following example shows that changing the order can change
the OS-QCP rank, even when the factors are chosen freely under
each order.

\begin{example}[Order-dependent quaternion CP rank]
	Consider \(\dot{\mathcal X}\in\Hq^{2\times2\times2}\) with
	\[
	\dot{\mathcal X}(:,:,1)
	=
	\begin{bmatrix}
		1 & \mathbf j\\
		\mathbf i & \mathbf k
	\end{bmatrix},
	\qquad
	\dot{\mathcal X}(:,:,2)
	=
	\begin{bmatrix}
		\mathbf k & \mathbf i\\
		-\mathbf j & -1
	\end{bmatrix}.
	\]
	Under order \(\pi=123\), the factors
	\(\dot{\mathbf a}^{(123)}=[1,\mathbf i]^T\),
	\(\dot{\mathbf b}^{(123)}=[1,\mathbf j]^T\), and
	\(\dot{\mathbf c}^{(123)}=[1,\mathbf k]^T\)
	satisfy
	\(\dot x_{ijk}
	=\dot a_i^{(123)}\dot b_j^{(123)}\dot c_k^{(123)}\).
	Since \(\dot{\mathcal X}\neq0\), \(R_{123}=1\).
	
	Suppose that \(\dot{\mathcal X}\) also admitted a rank-one
	representation under order \(\pi=213\), namely
	\(\dot x_{ijk}=\dot e_j\dot d_i\dot f_k\).
	On the first frontal slice, this gives
	\(\dot x_{ij1}=\dot e_j\dot h_i\), where
	\(\dot h_i=\dot d_i\dot f_1\).
	Since \(\dot x_{i11}\neq0\) for \(i=1,2\), both
	\(\dot e_1\) and \(\dot h_i\) are invertible. Thus,
	\[
	\dot x_{i21}\dot x_{i11}^{-1}
	=
	(\dot e_2\dot h_i)(\dot e_1\dot h_i)^{-1}
	=
	\dot e_2\dot h_i\dot h_i^{-1}\dot e_1^{-1}
	=
	\dot e_2\dot e_1^{-1},
	\]
	which is independent of \(i\). However,
	\(\dot x_{121}\dot x_{111}^{-1}=\mathbf j\), whereas
	\(\dot x_{221}\dot x_{211}^{-1}
	=\mathbf k(-\mathbf i)=-\mathbf j\).
	This contradiction rules out a rank-one representation under
	order \(\pi=213\).
	
	On the other hand, the factor matrices
	\[
	\dot{\mathbf A}^{(213)}
	=
	\begin{bmatrix}
		-\mathbf i & 0\\
		0 & -\mathbf i
	\end{bmatrix},
	\qquad
	\dot{\mathbf B}^{(213)}
	=
	\begin{bmatrix}
		\mathbf i & -1\\
		-\mathbf k & \mathbf j
	\end{bmatrix},
	\qquad
	\dot{\mathbf C}^{(213)}
	=
	\begin{bmatrix}
		1 & 1\\
		\mathbf k & \mathbf k
	\end{bmatrix}
	\]
	satisfy
	\(\dot x_{ijk}
	=\sum_{r=1}^{2}
	\dot b_{jr}^{(213)}\dot a_{ir}^{(213)}\dot c_{kr}^{(213)}\).
	Consequently, \(R_{123}=1\) and \(R_{213}=2\).
\end{example}

\begin{remark}[Reversed orders for pure quaternion tensors]
	\label{rem:reverse_order}
	For a pure quaternion tensor \(\dot{\mathcal X}\),
	which satisfies
	\(\overline{\dot{\mathcal X}}=-\dot{\mathcal X}\),
	reversed multiplication orders admit a direct
	correspondence.
	Let \(\pi^{\mathrm{rev}}\) denote the reverse of
	\(\pi\).
	If factors
	\(\dot{\mathbf A},\dot{\mathbf B},\dot{\mathbf C}\)
	generate an approximation
	\(\widehat{\dot{\mathcal X}}\) under order \(\pi\),
	then the factors
	\(-\overline{\dot{\mathbf A}},
	\overline{\dot{\mathbf B}},
	\overline{\dot{\mathbf C}}\)
	under order \(\pi^{\mathrm{rev}}\) generate
	\(-\overline{\widehat{\dot{\mathcal X}}}\), since
	\[
	\Pi_{\pi^{\mathrm{rev}}}
	\left(
	-\overline{\dot a},
	\overline{\dot b},
	\overline{\dot c}
	\right)
	=
	-\overline{
		\Pi_\pi(\dot a,\dot b,\dot c)
	}.
	\]
	This transformation is reversible and preserves
	the number of components and all factor Frobenius
	norms. Moreover,
	\[
	\left\|
	\dot{\mathcal X}
	+\overline{\widehat{\dot{\mathcal X}}}
	\right\|_F
	=
	\left\|
	\dot{\mathcal X}
	-\widehat{\dot{\mathcal X}}
	\right\|_F.
	\]
	Thus, the pairs
	\(123\leftrightarrow321\),
	\(132\leftrightarrow231\), and
	\(213\leftrightarrow312\)
	have equal order-specific ranks and equivalent
	least-squares approximation problems.
\end{remark}

\subsection{Essential Uniqueness and Order Dependence}

We first study uniqueness under a prescribed multiplication order
and distinguish it from equivalence between decompositions under
different orders. We then show that an essentially unique CP
decomposition of the complex representation does not eliminate
order dependence in OS-QCP.

Fix a multiplication order \(\pi\in S_3\).
Applying a common column permutation to
\(\dot{\mathbf A}^{(\pi)}\),
\(\dot{\mathbf B}^{(\pi)}\), and
\(\dot{\mathbf C}^{(\pi)}\)
only reorders the summands and leaves the represented tensor unchanged. Each component also admits a quaternion scaling indeterminacy.
For nonzero quaternions \(\dot u_r,\dot v_r\), multiply the first
factor column in the prescribed product on the right by \(\dot u_r\),
the middle factor column on the left by \(\dot u_r^{-1}\) and on
the right by \(\dot v_r\), and the last factor column on the left
by \(\dot v_r^{-1}\).
These transformations preserve each component by associativity:
\((\dot x\dot u_r)
(\dot u_r^{-1}\dot y\dot v_r)
(\dot v_r^{-1}\dot z)=\dot x\dot y\dot z\),
where \(\dot x,\dot y,\dot z\) denote its factor entries
in the prescribed multiplication order.

\begin{definition}[Essential uniqueness]
	\label{def:osqcp_essential_uniqueness}
	An order-\(\pi\) OS-QCP decomposition of
	\(\dot{\mathcal X}\) is essentially unique if every
	order-\(\pi\) decomposition of \(\dot{\mathcal X}\)
	with \(R_\pi\) components differs from it only by a
	common column permutation and the quaternion scaling
	transformations described above.
\end{definition}

A sufficient condition for essential uniqueness follows
from comparing matrix slices obtained by fixing the index
of the middle factor in the prescribed product.
For a fixed order \(\pi\), let \(\dot{\mathbf M}_t^{(\pi)}\)
denote the slice obtained by fixing the index of the middle
factor at \(t\). Its row and column indices are the indices
of the first and last factors in the prescribed product,
respectively.
Write \(\dot d_{tr}^{(\pi)}\) for the \((t,r)\) entry
of the middle factor matrix.
For instance, under order \(\pi=123\),
\(\dot d_{tr}^{(123)}=\dot b_{tr}^{(123)}\) and
\begin{equation}
	\dot{\mathbf M}_t^{(123)}
	=
	[\dot x_{itk}]_{i,k}
	=
	\dot{\mathbf A}^{(123)}
	\operatorname{diag}
	\left(
	\dot b_{t1}^{(123)},\ldots,
	\dot b_{tR_{123}}^{(123)}
	\right)
	\left(\dot{\mathbf C}^{(123)}\right)^T.
	\label{eq:ordered_factor_representation}
\end{equation}

\begin{theorem}[A sufficient condition for essential uniqueness]
	\label{thm:osqcp_essential_uniqueness}
	Consider an OS-QCP decomposition of a nonzero third-order
	quaternion tensor \(\dot{\mathcal X}\) under a prescribed
	multiplication order \(\pi\in S_3\).
	Suppose that there exist slice indices \(t_0,t_1\)
	such that
	\begin{equation}
		\operatorname{rank}_{\Hq}
		\left(\dot{\mathbf M}_{t_0}^{(\pi)}\right)
		=
		R_\pi,
		\label{eq:uniqueness_slice_rank}
	\end{equation}
	and the real numbers
	\begin{equation}
		\operatorname{Re}
		\left(
		\dot d_{t_1r}^{(\pi)}
		\left(\dot d_{t_0r}^{(\pi)}\right)^{-1}
		\right),
		\qquad r=1,\ldots,R_\pi,
		\label{eq:uniqueness_separation}
	\end{equation}
	are pairwise distinct.
	Then the decomposition is essentially unique under
	order \(\pi\).
	Here, \(\operatorname{rank}_{\Hq}\) denotes quaternion
	matrix rank; condition~\eqref{eq:uniqueness_slice_rank}
	ensures that all \(\dot d_{t_0r}^{(\pi)}\) are nonzero.
\end{theorem}

The proof of Theorem~\ref{thm:osqcp_essential_uniqueness}
is provided in Section~S1 of the supplementary materials. Essential uniqueness concerns decompositions under a fixed order.
It does not imply that decompositions under different orders
have equivalent factors. The following example establishes
this distinction even when the ranks are equal and both
decompositions are essentially unique.

\begin{example}[Equal ranks with inequivalent factors]
	\label{ex:equal_rank_inequivalent_orders}
	Set \(\dot p=1+\mathbf j\),
	\(\dot s=(-1+\mathbf j)/\sqrt{2}\), and
	\(\dot t=(\mathbf k-\mathbf i)/\sqrt{2}\).
	Consider \(\dot{\mathcal X}\in\Hq^{2\times2\times2}\) with
	\[
	\dot{\mathcal X}(:,:,1)
	=
	\begin{bmatrix}
		1 & \dot p\\
		\mathbf i & \mathbf i\dot p
	\end{bmatrix},
	\qquad
	\dot{\mathcal X}(:,:,2)
	=
	\begin{bmatrix}
		1 & -\dot p\\
		-\mathbf i & \mathbf i\dot p
	\end{bmatrix}.
	\]
	An order-\(123\) representation is given by
	\[
	\dot{\mathbf A}^{(123)}
	=
	\begin{bmatrix}
		1 & 1\\
		\mathbf i & -\mathbf i
	\end{bmatrix},
	\quad
	\dot{\mathbf B}^{(123)}
	=
	\begin{bmatrix}
		\mathbf i & \mathbf i\\
		\mathbf i-\mathbf k & -\mathbf i+\mathbf k
	\end{bmatrix},
	\quad
	\dot{\mathbf C}^{(123)}
	=
	\begin{bmatrix}
		-\mathbf i & 0\\
		0 & -\mathbf i
	\end{bmatrix}.
	\]
	The displayed representation gives \(R_{123}\leq2\).
	The slice
	\([\dot x_{i1k}]_{i,k}=\dot{\mathbf A}^{(123)}\)
	is invertible and hence has quaternion matrix rank two,
	which implies \(R_{123}\geq2\).
	Therefore, \(R_{123}=2\).
	For \(t_0=1\) and \(t_1=2\), the quantities
	\(\dot b_{21}^{(123)}(\dot b_{11}^{(123)})^{-1}
	=(\mathbf i-\mathbf k)(-\mathbf i)=\dot p\)
	and
	\(\dot b_{22}^{(123)}(\dot b_{12}^{(123)})^{-1}
	=(-\mathbf i+\mathbf k)(-\mathbf i)=-\dot p\)
	have distinct real parts \(1\) and \(-1\).
	Thus, Theorem~\ref{thm:osqcp_essential_uniqueness}
	establishes essential uniqueness under order \(123\).
	
	An order-\(213\) representation is given by
	\[
	\dot{\mathbf A}^{(213)}
	=
	\begin{bmatrix}
		1 & 1\\
		\dot s & -\dot s
	\end{bmatrix},
	\quad
	\dot{\mathbf B}^{(213)}
	=
	\begin{bmatrix}
		1 & 1\\
		-\sqrt{2}\mathbf k & \sqrt{2}\mathbf k
	\end{bmatrix},
	\quad
	\dot{\mathbf C}^{(213)}
	=
	\frac12
	\begin{bmatrix}
		1+\dot t & 1-\dot t\\
		1-\dot t & 1+\dot t
	\end{bmatrix}.
	\]
	This follows from
	\(\dot s\dot t=\mathbf i\),
	\(-\sqrt{2}\mathbf k\dot t=\dot p\), and
	\(-\sqrt{2}\mathbf k\dot s=\mathbf i\dot p\).
	The slice
	\([\dot x_{1jk}]_{j,k}
	=\begin{bmatrix}1&1\\ \dot p&-\dot p\end{bmatrix}\)
	is invertible, so \(R_{213}=2\).
	The middle factor \(\dot{\mathbf A}^{(213)}\) has
	first row \((1,1)\) and second-row entries with
	distinct real parts \(-1/\sqrt{2}\) and \(1/\sqrt{2}\).
	Theorem~\ref{thm:osqcp_essential_uniqueness} therefore
	gives essential uniqueness.
	
	In both orders, quaternion scaling left-multiplies
	each column of the last factor \(\dot{\mathbf C}^{(\pi)}\)
	by a nonzero quaternion, preserving its zero pattern.
	Every column of \(\dot{\mathbf C}^{(123)}\) contains
	one zero, whereas \(\dot{\mathbf C}^{(213)}\) has none.
	Hence, the factor triples cannot be made identical
	by their respective scaling transformations and column
	permutations. Essential uniqueness extends this
	conclusion to every pair of OS-QCP decompositions
	under these two orders.
\end{example}

A further question is whether essential uniqueness of the
CP decomposition of the complex representation eliminates
OS-QCP order dependence. Using the Cayley--Dickson representation, write
\(\dot q=z_1+z_2\mathbf j\), where
\(z_1=q_0+q_1\mathbf i\) and \(z_2=q_2+q_3\mathbf i\)
are complex numbers. The associated complex matrix
representation is \cite{zhang1997quaternions}
\begin{equation}
	\chi(\dot q)
	=
	\begin{bmatrix}
		z_1 & z_2\\
		-\overline{z_2} & \overline{z_1}
	\end{bmatrix}.
	\label{eq:entrywise_complex_representation}
\end{equation}
Applying \(\chi\) entry-wise to
\(\dot{\mathcal X}\in\Hq^{I_1\times I_2\times I_3}\)
yields
\(\mathcal Y\in
\mathbb C^{I_1\times I_2\times I_3\times2\times2}\),
where 
\(\mathcal Y(i,j,k,:,:)=\chi(\dot x_{ijk})\).
This representation determines \(\dot{\mathcal X}\)
uniquely and is independent of the multiplication order
used in its OS-QCP decomposition. The following example shows that essential uniqueness of
the complex CP decomposition of \(\mathcal Y\) does not
eliminate OS-QCP order dependence.
\begin{example}
	\label{ex:unique_complex_cp_order_dependence}
	Consider \(\dot{\mathcal X}\in\Hq^{2\times4\times2}\)
	with frontal slices
	\[
	\dot{\mathcal X}(:,:,1)
	=
	\begin{bmatrix}
		1 & \mathbf i & \mathbf j & \mathbf k\\
		\mathbf i & -1 & \mathbf k & -\mathbf j
	\end{bmatrix},
	\qquad
	\dot{\mathcal X}(:,:,2)
	=
	\begin{bmatrix}
		\mathbf i & -1 & -\mathbf k & \mathbf j\\
		-1 & -\mathbf i & \mathbf j & \mathbf k
	\end{bmatrix}.
	\]
	Its complex representation \(\mathcal Y\) has CP rank
	four and an essentially unique complex CP decomposition,
	whereas \(\dot{\mathcal X}\) has OS-QCP ranks
	\(R_{123}=1\) and \(R_{213}=2\).
	The explicit decompositions and their verification are
	provided in Section~S2 of the supplementary materials.
\end{example}

The underlying reason is that combining complex CP terms
into a quaternion rank-one component requires a product
in the prescribed order, with each factor depending only
on its own mode index.
Such a factorization may exist under one order but not
another.

\subsection{Approximation Properties}

For a fixed order \(\pi\) and a prescribed positive integer \(R\),
an OS-QCP approximation uses factor matrices $
\dot{\mathbf A}^{(\pi)}\in\Hq^{I_1\times R},
\dot{\mathbf B}^{(\pi)}\in\Hq^{I_2\times R}$, and $
\dot{\mathbf C}^{(\pi)}\in\Hq^{I_3\times R}$
to construct
\(\widehat{\dot{\mathcal X}}^{(\pi)}
\in\Hq^{I_1\times I_2\times I_3}\), whose entries are
\begin{equation}
	\widehat{\dot x}_{ijk}^{(\pi)}
	=
	\sum_{r=1}^{R}
	\Pi_{\pi}
	\left(
	\dot a_{ir}^{(\pi)},
	\dot b_{jr}^{(\pi)},
	\dot c_{kr}^{(\pi)}
	\right).
	\label{eq:rankR_osqcp_entry}
\end{equation}
The associated least-squares problem is
\begin{equation}
	\min_{
		\dot{\mathbf A}^{(\pi)},
		\dot{\mathbf B}^{(\pi)},
		\dot{\mathbf C}^{(\pi)}
	}
	\frac{1}{2}
	\left\|
	\dot{\mathcal X}
	-
	\widehat{\dot{\mathcal X}}^{(\pi)}
	\right\|_F^2.
	\label{eq:rankR_osqcp_approximation}
\end{equation}
The reconstructed tensor has order-\(\pi\) rank at most \(R\).

We next establish the existence of a best rank-one approximation
and of solutions to the regularized low-rank approximation problem.

\begin{theorem}[Best OS-QCP rank-one approximation]
	\label{thm:best_rank_one}
	For any nonzero tensor
	\(\dot{\mathcal X}\in\Hq^{I_1\times I_2\times I_3}\)
	and any fixed order \(\pi\in S_3\), the problem
	\begin{equation}
		\min_{
			\dot{\mathbf a}^{(\pi)},
			\dot{\mathbf b}^{(\pi)},
			\dot{\mathbf c}^{(\pi)}
		}
		\left\|
		\dot{\mathcal X}
		-
		\circ_{\Hq,\pi}
		\left(
		\dot{\mathbf a}^{(\pi)},
		\dot{\mathbf b}^{(\pi)},
		\dot{\mathbf c}^{(\pi)}
		\right)
		\right\|_F^2
		\label{eq:best_rank_one_osqcp}
	\end{equation}
	admits a global minimizer.
\end{theorem}

The proof of Theorem~\ref{thm:best_rank_one}
is provided in Section~S3 of the supplementary materials.

For \(R>1\), different components can cancel, so boundedness
of their sum does not imply boundedness of the individual
components. The rank-one existence result therefore does
not directly extend to this case.
In classical CP approximation, the set of tensors of rank at most
\(R\) can fail to be closed, and a best approximation may not
exist~\cite{de2008tensor}. This result does not directly establish
nonclosedness for OS-QCP, whose unregularized approximation problem
requires a separate existence analysis.

To guarantee existence, we introduce positive quadratic penalties
on all three factor matrices. For a prescribed order \(\pi\in S_3\)
and component count \(R\), consider
\begin{equation}
	\min_{
		\dot{\mathbf A}^{(\pi)},
		\dot{\mathbf B}^{(\pi)},
		\dot{\mathbf C}^{(\pi)}
	}
	\frac{1}{2}
	\left\|
	\dot{\mathcal X}
	-
	\widehat{\dot{\mathcal X}}^{(\pi)}
	\right\|_F^2
	+
	\frac{\lambda_A}{2}
	\|\dot{\mathbf A}^{(\pi)}\|_F^2
	+
	\frac{\lambda_B}{2}
	\|\dot{\mathbf B}^{(\pi)}\|_F^2
	+
	\frac{\lambda_C}{2}
	\|\dot{\mathbf C}^{(\pi)}\|_F^2,
	\label{eq:regularized_osqcp}
\end{equation}
where \(\widehat{\dot{\mathcal X}}^{(\pi)}\) is defined by
\eqref{eq:rankR_osqcp_entry}, the factor matrices satisfy
\(\dot{\mathbf A}^{(\pi)}\in\Hq^{I_1\times R}\),
\(\dot{\mathbf B}^{(\pi)}\in\Hq^{I_2\times R}\), and
\(\dot{\mathbf C}^{(\pi)}\in\Hq^{I_3\times R}\), and
\(\lambda_A,\lambda_B,\lambda_C>0\).

\begin{theorem}[Existence of a regularized OS-QCP minimizer]
	\label{thm:existence_regularized_osqcp}
	For any tensor
	\(\dot{\mathcal X}\in\Hq^{I_1\times I_2\times I_3}\),
	any prescribed order \(\pi\in S_3\), and any positive integer
	\(R\), problem \eqref{eq:regularized_osqcp} admits at least
	one global minimizer whenever
	\(\lambda_A,\lambda_B,\lambda_C>0\).
\end{theorem}

\begin{proof}
	Let \(F\) denote the objective in
	\eqref{eq:regularized_osqcp}, which is continuous in the
	real coordinates of the factor matrices.
	Set \(\lambda_{\min}=\min\{\lambda_A,\lambda_B,\lambda_C\}>0\).
	The nonnegative data-fitting term gives the lower bound
	\(F\geq\frac{\lambda_{\min}}{2}
	(\|\dot{\mathbf A}^{(\pi)}\|_F^2
	+\|\dot{\mathbf B}^{(\pi)}\|_F^2
	+\|\dot{\mathbf C}^{(\pi)}\|_F^2)\),
	so \(F\) is coercive.
	Continuity and coercivity in this finite-dimensional
	space ensure that \(F\) attains a global minimum.
\end{proof}

This existence guarantee requires all three regularization
parameters to be positive; it does not imply uniqueness or
convergence of an algorithm to a global minimizer.
The next section develops an ALS algorithm for computing
OS-QCP approximations under the canonical order \(\pi=123\).

\section{ALS Algorithm for Canonical OS-QCP}
\label{sec:als}

This section derives an ALS algorithm for the canonical order
\(\pi=123\), omitting the superscript \((123)\).
Updates for other prescribed orders follow analogously using
the corresponding left and right multiplication matrices.

\subsection{Least-Squares Formulation}

Let \(\dot{\mathcal X}\in\Hq^{I_1\times I_2\times I_3}\)
be nonzero, and let \(R\geq1\) be the prescribed component count.
The canonical OS-QCP approximation is
\begin{equation}
	\widehat{\dot x}_{ijk}
	=
	\sum_{r=1}^{R}
	\dot a_{ir}\dot b_{jr}\dot c_{kr},
	\label{eq:osqcp_reconstruction_canonical}
\end{equation}
where
\(\dot{\mathbf A}=[\dot a_{ir}]\in\Hq^{I_1\times R}\),
\(\dot{\mathbf B}=[\dot b_{jr}]\in\Hq^{I_2\times R}\), and
\(\dot{\mathbf C}=[\dot c_{kr}]\in\Hq^{I_3\times R}\).

With \(\lambda_A,\lambda_B,\lambda_C>0\), we solve
\eqref{eq:regularized_osqcp} under the canonical order \(\pi=123\).
The factors are updated sequentially, with the other two factors
held fixed during each update.
Optional proximal terms are included in the block subproblems;
setting their parameters to zero gives standard regularized ALS.

\subsection{Alternating Least-Squares Updates}

We first update \(\dot{\mathbf A}\) with
\(\dot{\mathbf B}\) and \(\dot{\mathbf C}\) fixed.
Set \(\dot q^A_{jkr}=\dot b_{jr}\dot c_{kr}\).
Then
\(\widehat{\dot x}_{ijk}
=\sum_{r=1}^{R}\dot a_{ir}\dot q^A_{jkr}\),
and the real matrix form of right quaternion multiplication gives
\begin{equation*}
	\operatorname{vec}_{\R}(\widehat{\dot x}_{ijk})
	=
	\sum_{r=1}^{R}
	\mathbf R(\dot q^A_{jkr})
	\operatorname{vec}_{\R}(\dot a_{ir}).
\end{equation*}
Define $
\mathbf D_A(j,k)
=
\begin{bmatrix}
	\mathbf R(\dot q^A_{jk1}) &
	\mathbf R(\dot q^A_{jk2}) &
	\cdots &
	\mathbf R(\dot q^A_{jkR})
\end{bmatrix}
\in\R^{4\times4R}$,
and let
\[
\mathbf z_i^A
=
\begin{bmatrix}
	\operatorname{vec}_{\R}(\dot a_{i1})^T &
	\cdots &
	\operatorname{vec}_{\R}(\dot a_{iR})^T
\end{bmatrix}^{T}
\in\R^{4R}.
\]
Thus,
\(\operatorname{vec}_{\R}(\widehat{\dot x}_{ijk})
=\mathbf D_A(j,k)\mathbf z_i^A\).

Vertically stack \(\mathbf D_A(j,k)\) over all pairs \((j,k)\),
with \(k\) varying fastest, to obtain
\(\mathbf D_A\in\R^{4I_2I_3\times4R}\).
For each fixed \(i\), stack
\(\operatorname{vec}_{\R}(\dot x_{ijk})\) in the same order
to form \(\mathbf y_i^A\in\R^{4I_2I_3}\).
Collect these vectors as $\mathbf Y_A
=
[\mathbf y_1^A,\ldots,\mathbf y_{I_1}^A]
\in\R^{4I_2I_3\times I_1}$ and $
\mathbf Z_A
=
[\mathbf z_1^A,\ldots,\mathbf z_{I_1}^A]
\in\R^{4R\times I_1}$.
The corresponding stacked reconstruction is
\(\mathbf D_A\mathbf Z_A\).
Since stacking real coordinates preserves the Frobenius norm,
the update is obtained from
\begin{equation}
	\min_{\mathbf Z_A}
	\frac{1}{2}
	\left\|
	\mathbf Y_A-\mathbf D_A\mathbf Z_A
	\right\|_F^2
	+
	\frac{\lambda_A}{2}
	\|\mathbf Z_A\|_F^2
	+
	\frac{\eta_A}{2}
	\left\|
	\mathbf Z_A-\mathbf Z_A^{\mathrm{old}}
	\right\|_F^2,
	\label{eq:A_subproblem}
\end{equation}
where \(\eta_A\geq0\) is a proximal parameter and
\(\mathbf Z_A^{\mathrm{old}}\) is the value before the
current block update.
The unique minimizer is
\begin{equation}
	\mathbf Z_A^{\mathrm{new}}
	=
	\left(
	\mathbf D_A^T\mathbf D_A
	+
	(\lambda_A+\eta_A)\mathbf I_{4R}
	\right)^{-1}
	\left(
	\mathbf D_A^T\mathbf Y_A
	+
	\eta_A\mathbf Z_A^{\mathrm{old}}
	\right).
	\label{eq:A_update_full}
\end{equation}

For the update of \(\dot{\mathbf B}\), the unknown entry
\(\dot b_{jr}\) appears between
\(\dot a_{ir}\) and \(\dot c_{kr}\).
Its real linear operator therefore involves both left and
right multiplication: $
\mathbf D_B(i,k)
=
\begin{bmatrix}
	\mathbf L(\dot a_{i1})\mathbf R(\dot c_{k1}) &
	\cdots &
	\mathbf L(\dot a_{iR})\mathbf R(\dot c_{kR})
\end{bmatrix}
\in\R^{4\times4R}$.
The matrix \(\mathbf D_B\in\R^{4I_1I_3\times4R}\) is formed
by vertically stacking these blocks over \((i,k)\), with \(k\)
varying fastest.
Let \(\mathbf Z_B\in\R^{4R\times I_2}\) have its \(j\)th
column formed by stacking the real coordinates of
\(\dot b_{j1},\ldots,\dot b_{jR}\).
The \(j\)th column of
\(\mathbf Y_B\in\R^{4I_1I_3\times I_2}\) stacks
\(\operatorname{vec}_{\R}(\dot x_{ijk})\) over \((i,k)\)
in the same order.
The resulting reconstruction is \(\mathbf D_B\mathbf Z_B\),
and the analogous regularized subproblem yields
\begin{equation}
	\mathbf Z_B^{\mathrm{new}}
	=
	\left(
	\mathbf D_B^T\mathbf D_B
	+
	(\lambda_B+\eta_B)\mathbf I_{4R}
	\right)^{-1}
	\left(
	\mathbf D_B^T\mathbf Y_B
	+
	\eta_B\mathbf Z_B^{\mathrm{old}}
	\right),
	\label{eq:B_update_full}
\end{equation}
where \(\eta_B\geq0\) and \(\mathbf Z_B^{\mathrm{old}}\)
denotes the value before the current block update.

For the update of \(\dot{\mathbf C}\), the unknown entry
\(\dot c_{kr}\) is left-multiplied by
\(\dot a_{ir}\dot b_{jr}\).
Define $
\mathbf D_C(i,j)
=
\begin{bmatrix}
	\mathbf L(\dot a_{i1}\dot b_{j1}) &
	\cdots &
	\mathbf L(\dot a_{iR}\dot b_{jR})
\end{bmatrix}
\in\R^{4\times4R}$.
The matrix \(\mathbf D_C\in\R^{4I_1I_2\times4R}\) is formed
by vertically stacking these blocks over \((i,j)\), with \(j\)
varying fastest.
Let \(\mathbf Z_C\in\R^{4R\times I_3}\) have its \(k\)th
column formed by stacking the real coordinates of
\(\dot c_{k1},\ldots,\dot c_{kR}\).
The \(k\)th column of
\(\mathbf Y_C\in\R^{4I_1I_2\times I_3}\) stacks
\(\operatorname{vec}_{\R}(\dot x_{ijk})\) over \((i,j)\)
in the same order.
The reconstruction is \(\mathbf D_C\mathbf Z_C\), and
the update is
\begin{equation}
	\mathbf Z_C^{\mathrm{new}}
	=
	\left(
	\mathbf D_C^T\mathbf D_C
	+
	(\lambda_C+\eta_C)\mathbf I_{4R}
	\right)^{-1}
	\left(
	\mathbf D_C^T\mathbf Y_C
	+
	\eta_C\mathbf Z_C^{\mathrm{old}}
	\right),
	\label{eq:C_update_full}
\end{equation}
where \(\eta_C\geq0\) and \(\mathbf Z_C^{\mathrm{old}}\)
denotes the value before the current block update.

Because \(\lambda_A,\lambda_B,\lambda_C>0\), the coefficient
matrices in \eqref{eq:A_update_full}--\eqref{eq:C_update_full}
are symmetric positive definite.
In implementation, the updates are computed by solving these
linear systems rather than explicitly forming the inverses.
Each coefficient matrix is factorized once and used for all
right-hand sides in the corresponding block update.

The quaternion factor matrix is recovered by grouping every four
consecutive real coordinates within each column of
\(\mathbf Z_A\), \(\mathbf Z_B\), or \(\mathbf Z_C\)
into a quaternion entry.
The factors are updated sequentially in the order
\(\dot{\mathbf A}\), \(\dot{\mathbf B}\), and
\(\dot{\mathbf C}\), using the latest available values.
In particular, \(\mathbf D_B\) uses the updated
\(\dot{\mathbf A}\), and \(\mathbf D_C\) uses the updated
\(\dot{\mathbf A}\) and \(\dot{\mathbf B}\).

Let \(F^{(t)}\) denote the value of the regularized objective
in \eqref{eq:regularized_osqcp} after the \(t\)-th full iteration.
The iteration is terminated when
\begin{equation}
	\frac{
		|F^{(t)}-F^{(t-1)}|
	}{
		\max\{1,F^{(t-1)}\}
	}
	<
	\varepsilon,
	\label{eq:osqcp_stopping}
\end{equation}
or when the maximum number of iterations is reached,
where \(\varepsilon>0\) is a prescribed tolerance.
This criterion monitors objective stagnation; it does not
by itself certify stationarity.
The relative reconstruction error
\(e^{(t)}
=\|\dot{\mathcal X}-\widehat{\dot{\mathcal X}}^{(t)}\|_F
/\|\dot{\mathcal X}\|_F\)
can additionally be recorded to assess the approximation quality.

The overall procedure is summarized in Algorithm~\ref{alg:canonical_osqcp}.

\begin{algorithm}[h]
	\caption{ALS for regularized canonical OS-QCP approximation}
	\label{alg:canonical_osqcp}
	\begin{algorithmic}[1]
		\STATE \textbf{Input:} Nonzero quaternion tensor
		\(\dot{\mathcal X}\), component count \(R\geq1\),
		regularization parameters \(\lambda_A,\lambda_B,\lambda_C>0\),
		proximal parameters \(\eta_A,\eta_B,\eta_C\geq0\),
		tolerance \(\varepsilon>0\), and maximum number of
		iterations \(T_{\max}\).
		\STATE Initialize nonzero factor matrices
		\(\dot{\mathbf A}\), \(\dot{\mathbf B}\), and
		\(\dot{\mathbf C}\), compute \(F^{(0)}\), and set \(t=0\).
		\REPEAT
		\STATE Set \(t\gets t+1\).
		\STATE Update \(\dot{\mathbf A}\) by
		\eqref{eq:A_update_full}.
		\STATE Update \(\dot{\mathbf B}\) by
		\eqref{eq:B_update_full}, using the updated
		\(\dot{\mathbf A}\).
		\STATE Update \(\dot{\mathbf C}\) by
		\eqref{eq:C_update_full}, using the updated
		\(\dot{\mathbf A}\) and \(\dot{\mathbf B}\).
		\STATE Compute
		\(\widehat{\dot{\mathcal X}}^{(t)}\) and \(F^{(t)}\).
		\UNTIL{criterion \eqref{eq:osqcp_stopping} is satisfied
			or \(t=T_{\max}\).}
		\STATE \textbf{Output:}
		\(\dot{\mathbf A}\), \(\dot{\mathbf B}\), and
		\(\dot{\mathbf C}\).
	\end{algorithmic}
\end{algorithm}

\subsection{Convergence and Complexity Remarks}

The objective in \eqref{eq:regularized_osqcp} is jointly
nonconvex in the three factor matrices.
Each block update minimizes its regularized least-squares
subproblem, optionally augmented by a proximal term.
The previous block value is feasible and has zero proximal
penalty, so each exact update does not increase the original
regularized objective.
Consequently, the objective values converge because they are
nonincreasing and bounded below by zero.
The positive quadratic penalties also ensure that the factor
iterates remain bounded.
Convergence of the factors and global optimality are not
guaranteed by this argument.

We analyze the computational cost of explicitly forming the
design matrices and solving the normal equations, with one
matrix factorization per block update.
Ignoring constant factors associated with the four real
coordinates of a quaternion, forming
\(\mathbf D_A^T\mathbf D_A\) costs \(O(I_2I_3R^2)\),
and computing \(\mathbf D_A^T\mathbf Y_A\) costs
\(O(I_1I_2I_3R)\).
Factorizing the coefficient matrix and solving for all \(I_1\)
right-hand sides cost \(O(R^3+I_1R^2)\).
Thus, updating \(\dot{\mathbf A}\) costs
\(O(I_2I_3R^2+I_1I_2I_3R+R^3+I_1R^2)\).
The costs for \(\dot{\mathbf B}\) and \(\dot{\mathbf C}\)
follow by permuting the mode dimensions.
One full ALS iteration therefore costs
\[
O\!\left(
I_1I_2I_3R
+
(I_1I_2+I_1I_3+I_2I_3)R^2
+
R^3
\right).
\]
When \(R\) is much smaller than all three tensor dimensions,
the dominant cost is \(O(I_1I_2I_3R)\).

\section{Order-Ensemble Quaternion CP for Tensor Completion}
\label{sec:completion}

In OS-QCP, the \(r\)th component has entries
\(\Pi_\pi(\dot a_{ir},\dot b_{jr},\dot c_{kr})\),
where the multiplication order \(\pi\) is prescribed.
The factors are determined from the data, but the
multiplication rule is fixed in advance.
In applications such as tensor completion, however,
there may be no clear reason to favor one order
or to restrict the representation to a single order. A natural question is whether the decomposition model can learn
a combination of multiplication orders adaptively
from the data.

This motivates a new CP-type representation, termed
order-ensemble quaternion CP (OE-QCP).
For a third-order quaternion tensor, OE-QCP also seeks
three quaternion factor matrices.
Whereas OS-QCP determines these factors under a single
prescribed multiplication order, OE-QCP determines them
using a data-adaptively weighted sum of the six ordered
products of their entries within each component.
We first define the OE-QCP decomposition and then
formulate the corresponding tensor completion model.

\subsection{Order-Ensemble Quaternion CP Decomposition}

Let
\(\boldsymbol{\omega}
=(\omega_{123},\allowbreak
\omega_{132},\allowbreak
\omega_{213},\allowbreak
\omega_{231},\allowbreak
\omega_{312},\allowbreak
\omega_{321})^T\)
be the order-weight vector, constrained to the simplex
\[
\Delta_6
=
\left\{
\boldsymbol{\omega}\in\R^6:
\boldsymbol{\omega}\geq0,\quad
\mathbf{1}^T\boldsymbol{\omega}=1
\right\}.
\]

\begin{definition}[Order-ensemble quaternion CP decomposition]
	\label{def:oe_qcp}
	Consider a third-order quaternion tensor
	\(\dot{\mathcal X}\in\Hq^{I_1\times I_2\times I_3}\).
	An OE-QCP decomposition of \(\dot{\mathcal X}\) is defined as
	\begin{equation}
		\dot x_{ijk}
		=
		\sum_{r=1}^{R}
		\sum_{\pi\in S_3}
		\omega_\pi
		\Pi_\pi
		\left(
		\dot a_{ir},
		\dot b_{jr},
		\dot c_{kr}
		\right),
		\qquad
		\boldsymbol{\omega}\in\Delta_6,
		\label{eq:oeqcp_decomposition}
	\end{equation}
	where $
	\dot{\mathbf A}\in\Hq^{I_1\times R},
	\dot{\mathbf B}\in\Hq^{I_2\times R}$, and $
	\dot{\mathbf C}\in\Hq^{I_3\times R}$
	are quaternion-valued factor matrices associated with the
	first, second, and third tensor modes, respectively.
	
	The smallest nonnegative integer \(R\) for which
	\eqref{eq:oeqcp_decomposition} holds, with both the factors
	and \(\boldsymbol{\omega}\in\Delta_6\) allowed to vary,
	is called the OE-QCP rank, denoted by
	\(R_{\mathrm{OE}}
	:=\operatorname{rank}_{\mathrm{OE}}(\dot{\mathcal X})\).
	The zero tensor has OE-QCP rank zero, with the empty
	sum interpreted as zero.
\end{definition}

\begin{remark}[Relationship with OS-QCP and special cases]
	\begin{itemize}
		\item
		OE-QCP defines a new CP-type decomposition whose
		\(r\)th component has entries
		\(\sum_{\pi\in S_3}\omega_\pi
		\Pi_\pi(\dot a_{ir},\dot b_{jr},\dot c_{kr})\).
		The factor matrices are determined under this weighted
		product rule, avoiding the need to prescribe a single
		multiplication order in advance.
		Accordingly, \(R_{\mathrm{OE}}\) is the smallest number
		of such weighted components needed to represent
		\(\dot{\mathcal X}\) exactly.
		\item
		If \(\omega_\pi=1\) for a prescribed order
		\(\pi\in S_3\) and all other weights vanish,
		OE-QCP reduces to the corresponding OS-QCP model. With the weights fixed in this way, the smallest
		number of components required for an exact representation
		equals the corresponding order-specific OS-QCP rank
		\(R_\pi\).
		If all factor entries are real-valued or belong
		to a common complex subfield of \(\Hq\), the six
		ordered products coincide. Since
		\(\sum_{\pi\in S_3}\omega_\pi=1\),
		OE-QCP then reduces to the corresponding real or
		complex CP representation, independently of
		\(\boldsymbol{\omega}\).
		
		\item
		Even for a pure quaternion target tensor, the
		reversal relation in Remark~\ref{rem:reverse_order}
		requires a transformation of the factors and does
		not equate reversed products formed from the same
		factors. It therefore does not justify merging
		reversed terms by adding their weights.
		OE-QCP retains all six orders, with their
		contributions controlled by the order weights.
	\end{itemize}
\end{remark}

%For each component \(r\), the three factor columns generate
%a tensor with entries
%\(\sum_{\pi\in S_3}\omega_\pi
%\Pi_\pi(\dot a_{ir},\dot b_{jr},\dot c_{kr})\).
%Thus, OE-QCP retains a sum-of-components structure while
%replacing each single ordered product by a weighted combination
%of ordered products. An OE-QCP component generally cannot be represented by a single ordered product under any prescribed multiplication order.
%
%\begin{remark}[Special cases of OE-QCP]
%	If \(\omega_\pi=1\) for a prescribed order
%	\(\pi\in S_3\) and all other weights vanish, OE-QCP reduces
%	to the corresponding OS-QCP model. If all factor entries are real-valued, or belong to a common
%	complex subfield of \(\Hq\), the six ordered products coincide.
%	Since the order weights form a simplex,
%	\(\sum_{\pi\in S_3}\omega_\pi=1\),
%	OE-QCP then reduces to the corresponding real or complex CP
%	representation, regardless of the choice of
%	\(\boldsymbol{\omega}\). In addition, note that even for a pure quaternion target tensor, the
%	reversal relation in Remark~\ref{rem:reverse_order} requires
%	a transformation of the factors and does not imply equality
%	of reversed products formed from the same factors.
%	Thus, it does not justify merging reversed terms in OE-QCP
%	by adding their weights.
%	All six orders are therefore retained even for pure quaternion
%	target tensors, with their contributions controlled by
%	the order weights.
%\end{remark}

\begin{proposition}[Existence of OE-QCP representation]
	\label{prop:oeqcp_rank_relation}
	Every tensor
	\(\dot{\mathcal X}\in\Hq^{I_1\times I_2\times I_3}\)
	admits an exact OE-QCP representation with finitely many
	components for any fixed
	\(\boldsymbol{\omega}\in\Delta_6\).
	Moreover,
	\begin{equation}
		0
		\leq
		\operatorname{rank}_{\mathrm{OE}}
		(\dot{\mathcal X})
		\leq
		\min_{\pi\in S_3}
		\operatorname{rank}_{\Hq,\pi}
		(\dot{\mathcal X})
		\leq
		I_1I_2I_3 .
		\label{eq:oeqcp_rank_bound}
	\end{equation}
\end{proposition}

\begin{proof}
	For each nonzero entry \(\dot x_{ijk}\), choose three
	vectors supported only at indices \(i\), \(j\), and \(k\),
	with corresponding entries \(\dot x_{ijk}\), \(1\), and \(1\).
	All six ordered outer products coincide and have
	\(\dot x_{ijk}\) at \((i,j,k)\) and zero elsewhere.
	Since the weights sum to one, their weighted combination
	gives the same tensor.
	Summing over the nonzero entries yields an exact
	OE-QCP representation with at most \(I_1I_2I_3\)
	components for any fixed \(\boldsymbol{\omega}\in\Delta_6\).
	
	Choosing \(\omega_\pi=1\) and all other weights zero
	recovers the order-\(\pi\) OS-QCP representation.
	Thus,
	\(\operatorname{rank}_{\mathrm{OE}}(\dot{\mathcal X})
	\leq\operatorname{rank}_{\Hq,\pi}(\dot{\mathcal X})\)
	for every \(\pi\in S_3\), proving
	\eqref{eq:oeqcp_rank_bound}.
\end{proof}

For subsequent approximation and completion problems, the component
count \(R\) is prescribed as an estimate of the unknown OE-QCP rank.
The factor matrices and order weights are jointly estimated according
to
\begin{equation}
	\widehat{\dot{x}}_{ijk}^{\mathrm{OE}}
	=
	\sum_{\pi\in S_3}
	\omega_\pi
	\sum_{r=1}^{R}
	\Pi_\pi
	(
	\dot a_{ir},
	\dot b_{jr},
	\dot c_{kr}
	).
	\label{eq:oeqcp_model}
\end{equation}
Figure~\ref{fig:oeqcp_framework} illustrates the entrywise structure
of the OE-QCP representation.

\begin{figure}[h]
	\centering
	\resizebox{\linewidth}{!}{%
		\begin{tikzpicture}[
			font=\small,
			>=Stealth,
			line width=0.8pt,
			box/.style={
				draw=black,
				fill=white,
				rounded corners=3pt,
				align=center,
				inner sep=7pt
			},
			product/.style={
				box,
				minimum width=3.35cm,
				minimum height=0.68cm
			},
			flow/.style={
				->,
				draw=black,
				line width=0.8pt
			}
			]
			
			\node[box, minimum width=2.8cm] (shared) at (0,0)
			{Factor entries\\[7pt]
				\begin{tabular}{@{}c@{\qquad}c@{}}
					Mode \(1\) & \(\dot a_{ir}\)\\[7pt]
					Mode \(2\) & \(\dot b_{jr}\)\\[7pt]
					Mode \(3\) & \(\dot c_{kr}\)
			\end{tabular}};
			
			\node[product] (p123) at (4.5,2.25)
			{\(\Pi_{123}:\;
				\dot a_{ir}\dot b_{jr}\dot c_{kr}\)};
			
			\node[product] (p132) at (4.5,1.35)
			{\(\Pi_{132}:\;
				\dot a_{ir}\dot c_{kr}\dot b_{jr}\)};
			
			\node[product] (p213) at (4.5,0.45)
			{\(\Pi_{213}:\;
				\dot b_{jr}\dot a_{ir}\dot c_{kr}\)};
			
			\node[product] (p231) at (4.5,-0.45)
			{\(\Pi_{231}:\;
				\dot b_{jr}\dot c_{kr}\dot a_{ir}\)};
			
			\node[product] (p312) at (4.5,-1.35)
			{\(\Pi_{312}:\;
				\dot c_{kr}\dot a_{ir}\dot b_{jr}\)};
			
			\node[product] (p321) at (4.5,-2.25)
			{\(\Pi_{321}:\;
				\dot c_{kr}\dot b_{jr}\dot a_{ir}\)};
			
			\node[
			draw=black,
			fill=white,
			circle,
			minimum size=1.15cm,
			inner sep=2pt
			] (weighted) at (8.5,0)
			{\(\displaystyle\sum\)};
			
			\node[align=center] at (8.5,-1)
			{Weighted sum\\over orders};
			
			\coordinate (weightbase) at
			(weighted.center |- p123.north);
			\node[box, anchor=south]
			(weights) at ([yshift=-0.2cm]weightbase)
			{Order weights \(\boldsymbol{\omega}\)\\[4pt]
				\(\displaystyle
				\omega_\pi\geq0,\qquad
				\sum_{\pi\in S_3}\omega_\pi=1
				\)};
			
			\node[box, minimum width=1.65cm] (sumr) at (11,0)
			{Sum over \(r\)\\[4pt]
				\(\displaystyle\sum_{r=1}^{R}\)};
			
			\node (output) at (13,0)
			{\(\widehat{\dot x}_{ijk}^{\mathrm{OE}}\)};
			
			\draw[flow]
			(shared.east) to[out=0,in=180] (p123.west);
			\draw[flow]
			(shared.east) to[out=0,in=180] (p132.west);
			\draw[flow]
			(shared.east) to[out=0,in=180] (p213.west);
			\draw[flow]
			(shared.east) to[out=0,in=180] (p231.west);
			\draw[flow]
			(shared.east) to[out=0,in=180] (p312.west);
			\draw[flow]
			(shared.east) to[out=0,in=180] (p321.west);
			
			\draw[flow]
			(p123.east) to[out=0,in=150]
			node[pos=0.42,above,fill=white,inner sep=1pt]
			{\(\omega_{123}\)} (weighted.150);
			
			\draw[flow]
			(p132.east) to[out=0,in=165]
			node[pos=0.42,above,fill=white,inner sep=1pt]
			{\(\omega_{132}\)} (weighted.165);
			
			\draw[flow]
			(p213.east) to[out=0,in=175]
			node[pos=0.42,above,fill=white,inner sep=1pt]
			{\(\omega_{213}\)} (weighted.175);
			
			\draw[flow]
			(p231.east) to[out=0,in=185]
			node[pos=0.42,below,fill=white,inner sep=1pt]
			{\(\omega_{231}\)} (weighted.185);
			
			\draw[flow]
			(p312.east) to[out=0,in=195]
			node[pos=0.42,below,fill=white,inner sep=1pt]
			{\(\omega_{312}\)} (weighted.195);
			
			\draw[flow]
			(p321.east) to[out=0,in=210]
			node[pos=0.42,below,fill=white,inner sep=1pt]
			{\(\omega_{321}\)} (weighted.210);
			
			\draw[flow] (weights.south) -- (weighted.north);
			\draw[flow] (weighted.east) -- (sumr.west);
			\draw[flow] (sumr.east) -- (output.west);
			
		\end{tikzpicture}%
	}
	\caption{Entrywise structure of the OE-QCP representation.
		Each component combines the six ordered products of
		three quaternion factor entries using the order weights
		\(\boldsymbol{\omega}\).
	}
	\label{fig:oeqcp_framework}
\end{figure}
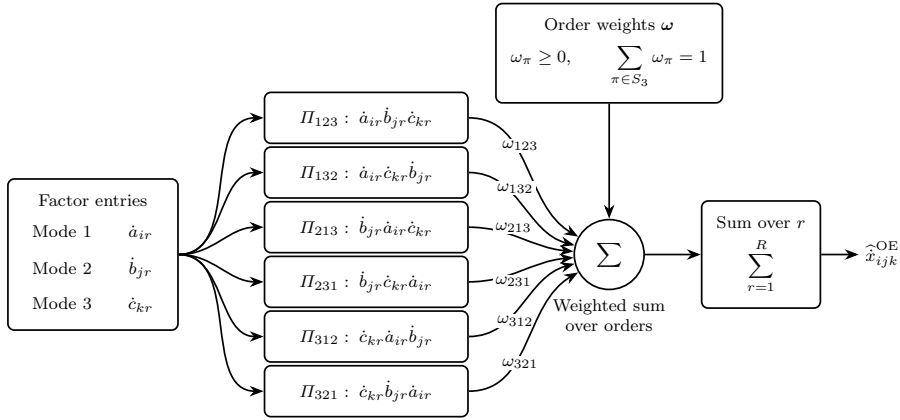

For fixed weights, the OE-QCP representation is linear over
\(\mathbb R\) in each factor matrix when the other two are fixed.
For fixed factor matrices, it is linear in the order-weight vector.
These properties enable real least-squares updates for the factor
matrices and a simplex-constrained least-squares update for the
weights. The completion model below exploits this structure.

\subsection{Quaternion Tensor Completion Model}

Let \(\dot{\mathcal Y}\in\Hq^{I_1\times I_2\times I_3}\)
be a partially observed quaternion tensor.
Let \(\Omega\) denote the observed index set and
\(\Omega^c\) its complement, corresponding to the missing entries.
The sampling operator \(\mathcal P_\Omega\) retains entries
indexed by \(\Omega\) and sets all other entries to zero.
For color video completion, each RGB pixel is encoded as
a pure quaternion. The three color channels are fitted
jointly through the quaternion-valued OE-QCP representation.

For a prescribed component count \(R\), the completion
problem is formulated as
\begin{equation}
	\begin{aligned}
		\min_{\dot{\mathbf{A}},\dot{\mathbf{B}},
			\dot{\mathbf{C}},\boldsymbol{\omega}}
		\
		&
		\frac{1}{2}
		\left\|
		\mathcal{P}_{\Omega}
		\left(
		\widehat{\dot{\mathcal{X}}}^{\mathrm{OE}}
		-
		\dot{\mathcal{Y}}
		\right)
		\right\|_F^2
		\!\!+
		\frac{\lambda_A}{2}\|\dot{\mathbf{A}}\|_F^2
		+
		\frac{\lambda_B}{2}\|\dot{\mathbf{B}}\|_F^2
		+
		\frac{\lambda_C}{2}\|\dot{\mathbf{C}}\|_F^2
		+
		\frac{\lambda_{\omega}}{2}
		\|\boldsymbol{\omega}\|_2^2,
		\\
		\mathrm{s.t.}
		\quad
		&
		\boldsymbol{\omega}\geq0,
		\qquad
		\mathbf{1}^T\boldsymbol{\omega}=1.
	\end{aligned}
	\label{eq:oeqcp_completion_model}
\end{equation}
Here,
\(\widehat{\dot{\mathcal{X}}}^{\mathrm{OE}}\)
is defined by \eqref{eq:oeqcp_model}, and
\(\lambda_A,\lambda_B,\lambda_C,\lambda_{\omega}>0\).
The first term fits the observed entries, while the
remaining terms regularize the factors and weights.
On the simplex, the weight regularizer favors more
balanced weights rather than sparse order selection.

\begin{proposition}[Existence of a regularized OE-QCP approximation]
	\label{prop:oeqcp_existence}
	For any observed tensor, observation set, and prescribed
	component count \(R\), problem
	\eqref{eq:oeqcp_completion_model} admits a global minimizer.
	The same conclusion holds when
	\(\boldsymbol{\omega}\in\Delta_6\) is fixed.
\end{proposition}

\begin{proof}
	The objective is continuous in the real coordinates of
	the factors and weights. Its positive quadratic factor
	regularizers imply that every sublevel set has bounded
	factor matrices. The weights belong to the compact set
	\(\Delta_6\). Hence, every nonempty sublevel set is
	closed and bounded in a finite-dimensional real space.
	The objective therefore attains its minimum.
	The argument also applies with fixed weights.
\end{proof}

After optimization, the completed tensor is obtained by
retaining the observed entries and filling the missing
entries with the OE-QCP reconstruction:
\begin{equation}
	\dot{\mathcal{X}}_{\mathrm{rec}}
	=
	\mathcal{P}_{\Omega}(\dot{\mathcal{Y}})
	+
	\mathcal{P}_{\Omega^c}
	\left(
	\widehat{\dot{\mathcal{X}}}^{\mathrm{OE}}
	\right).
	\label{eq:completed_tensor_output}
\end{equation}
For RGB data, the three imaginary components provide
the recovered color channels.

\subsection{Optimization Algorithm}

Problem~\eqref{eq:oeqcp_completion_model} is solved by
alternating updates of the three factor matrices and
the order weights. The factor updates use the real matrix
forms of quaternion left and right multiplication,
as in OS-QCP. The only change is that each design block
combines the six order-specific operators with their weights.

The factor blocks are updated sequentially, using the latest
available values of the other factors. Within each block,
the row subproblems are independent and can be solved
in parallel. Optional proximal terms with fixed parameters
\(\eta_A,\eta_B,\eta_C\geq0\) can be included;
setting them to zero gives the ordinary regularized ALS updates.

For fixed
\(\boldsymbol{\omega}\),
\(\dot{\mathbf B}\), and
\(\dot{\mathbf C}\), define
\begin{equation*}
	\begin{aligned}
		\mathbf{M}_{A,r}(j,k)
		={}&
		\omega_{123}\mathbf{R}(\dot{b}_{jr}\dot{c}_{kr})
		+
		\omega_{132}\mathbf{R}(\dot{c}_{kr}\dot{b}_{jr})
		+
		\omega_{213}\mathbf{L}(\dot{b}_{jr})
		\mathbf{R}(\dot{c}_{kr})
		\\
		&+
		\omega_{231}\mathbf{L}(\dot{b}_{jr}\dot{c}_{kr})
		+
		\omega_{312}\mathbf{L}(\dot{c}_{kr})
		\mathbf{R}(\dot{b}_{jr})
		+
		\omega_{321}\mathbf{L}(\dot{c}_{kr}\dot{b}_{jr}).
	\end{aligned}
\end{equation*}
The design block associated with \((j,k)\) is
\begin{equation}
	\mathbf{D}_A^{\mathrm{OE}}(j,k)
	=
	\begin{bmatrix}
		\mathbf{M}_{A,1}(j,k) &
		\mathbf{M}_{A,2}(j,k) &
		\cdots &
		\mathbf{M}_{A,R}(j,k)
	\end{bmatrix}
	\in
	\R^{4\times 4R}.
	\label{eq:DA_OE}
\end{equation}
For each \(i\), define
\(\Omega_i^A=\{(j,k):(i,j,k)\in\Omega\}\).
Let \(\mathbf z_i^A\in\R^{4R}\) be obtained by stacking
\(\operatorname{vec}_{\R}(\dot a_{i1}),\ldots,
\operatorname{vec}_{\R}(\dot a_{iR})\).
Then
\(\operatorname{vec}_{\R}
(\widehat{\dot x}_{ijk}^{\mathrm{OE}})
=\mathbf D_A^{\mathrm{OE}}(j,k)\mathbf z_i^A\),
and the normal equation for its update is
\begin{equation}
	\mathbf{N}_i^A\mathbf{z}_i^A
	=
	\mathbf{h}_i^A,
	\label{eq:A_normal_equation_OE}
\end{equation}
where
\begin{equation*}
	\begin{aligned}
		\mathbf{N}_i^A
		={}&
		\sum_{(j,k)\in\Omega_i^A}
		\bigl(\mathbf{D}_A^{\mathrm{OE}}(j,k)\bigr)^T
		\mathbf{D}_A^{\mathrm{OE}}(j,k)
		+
		(\lambda_A+\eta_A)\mathbf{I}_{4R},
		\\
		\mathbf{h}_i^A
		={}&
		\sum_{(j,k)\in\Omega_i^A}
		\bigl(\mathbf{D}_A^{\mathrm{OE}}(j,k)\bigr)^T
		\operatorname{vec}_{\R}(\dot{y}_{ijk})
		+
		\eta_A(\mathbf{z}_i^A)^{\mathrm{old}}.
	\end{aligned}
\end{equation*}
Here, \((\mathbf z_i^A)^{\mathrm{old}}\) denotes the row
before the current block update.
With the other variables fixed, all row subproblems for
\(i=1,\ldots,I_1\) are mutually independent.

Similarly, for updating \(\dot{\mathbf B}\), define
\begin{equation*}
	\begin{aligned}
		\mathbf{M}_{B,r}(i,k)
		={}&
		\omega_{123}\mathbf{L}(\dot{a}_{ir})
		\mathbf{R}(\dot{c}_{kr})
		+
		\omega_{132}\mathbf{L}(\dot{a}_{ir}\dot{c}_{kr})
		+
		\omega_{213}\mathbf{R}(\dot{a}_{ir}\dot{c}_{kr})
		\\
		&+
		\omega_{231}\mathbf{R}(\dot{c}_{kr}\dot{a}_{ir})
		+
		\omega_{312}\mathbf{L}(\dot{c}_{kr}\dot{a}_{ir})
		+
		\omega_{321}\mathbf{L}(\dot{c}_{kr})
		\mathbf{R}(\dot{a}_{ir}).
	\end{aligned}
\end{equation*}
The block row
\(\mathbf D_B^{\mathrm{OE}}(i,k)\in\R^{4\times4R}\)
is formed by concatenating
\(\mathbf M_{B,r}(i,k)\) over \(r=1,\ldots,R\).
For each \(j\), define
\(\Omega_j^B=\{(i,k):(i,j,k)\in\Omega\}\), and let
\(\mathbf z_j^B\in\R^{4R}\) stack the real coordinate
vectors of the entries in the \(j\)th row of
\(\dot{\mathbf B}\).
Its update satisfies
\begin{equation}
	\mathbf{N}_j^B\mathbf{z}_j^B
	=
	\mathbf{h}_j^B,
	\label{eq:B_normal_equation_OE}
\end{equation}
where
\begin{equation*}
	\begin{aligned}
		\mathbf{N}_j^B
		={}&
		\sum_{(i,k)\in\Omega_j^B}
		\bigl(\mathbf{D}_B^{\mathrm{OE}}(i,k)\bigr)^T
		\mathbf{D}_B^{\mathrm{OE}}(i,k)
		+
		(\lambda_B+\eta_B)\mathbf{I}_{4R},
		\\
		\mathbf{h}_j^B
		={}&
		\sum_{(i,k)\in\Omega_j^B}
		\bigl(\mathbf{D}_B^{\mathrm{OE}}(i,k)\bigr)^T
		\operatorname{vec}_{\R}(\dot{y}_{ijk})
		+
		\eta_B(\mathbf{z}_j^B)^{\mathrm{old}}.
	\end{aligned}
\end{equation*}
The updates for \(j=1,\ldots,I_2\) are independent
and can be computed in parallel.

For updating \(\dot{\mathbf C}\), define
\begin{equation*}
	\begin{aligned}
		\mathbf{M}_{C,r}(i,j)
		={}&
		\omega_{123}\mathbf{L}(\dot{a}_{ir}\dot{b}_{jr})
		+
		\omega_{132}\mathbf{L}(\dot{a}_{ir})
		\mathbf{R}(\dot{b}_{jr})
		+
		\omega_{213}\mathbf{L}(\dot{b}_{jr}\dot{a}_{ir})
		\\
		&+
		\omega_{231}\mathbf{L}(\dot{b}_{jr})
		\mathbf{R}(\dot{a}_{ir})
		+
		\omega_{312}\mathbf{R}(\dot{a}_{ir}\dot{b}_{jr})
		+
		\omega_{321}\mathbf{R}(\dot{b}_{jr}\dot{a}_{ir}).
	\end{aligned}
\end{equation*}
The block row
\(\mathbf D_C^{\mathrm{OE}}(i,j)\in\R^{4\times4R}\)
is formed by concatenating
\(\mathbf M_{C,r}(i,j)\) over \(r=1,\ldots,R\).
For each \(k\), define
\(\Omega_k^C=\{(i,j):(i,j,k)\in\Omega\}\), and let
\(\mathbf z_k^C\in\R^{4R}\) stack the real coordinate
vectors of the entries in the \(k\)th row of
\(\dot{\mathbf C}\).
Its update is given by
\begin{equation}
	\mathbf{N}_k^C\mathbf{z}_k^C
	=
	\mathbf{h}_k^C,
	\label{eq:C_normal_equation_OE}
\end{equation}
where
\begin{equation*}
	\begin{aligned}
		\mathbf{N}_k^C
		={}&
		\sum_{(i,j)\in\Omega_k^C}
		\bigl(\mathbf{D}_C^{\mathrm{OE}}(i,j)\bigr)^T
		\mathbf{D}_C^{\mathrm{OE}}(i,j)
		+
		(\lambda_C+\eta_C)\mathbf{I}_{4R},
		\\
		\mathbf{h}_k^C
		={}&
		\sum_{(i,j)\in\Omega_k^C}
		\bigl(\mathbf{D}_C^{\mathrm{OE}}(i,j)\bigr)^T
		\operatorname{vec}_{\R}(\dot{y}_{ijk})
		+
		\eta_C(\mathbf{z}_k^C)^{\mathrm{old}}.
	\end{aligned}
\end{equation*}
All row subproblems for \(k=1,\ldots,I_3\) are independent.
Since \(\lambda_A,\lambda_B,\lambda_C>0\),
each normal matrix is positive definite even when
the corresponding row contains no observed entries.

For fixed factor matrices, the weights are updated through
a six-dimensional constrained least-squares problem.
Let \(\operatorname{vec}_{\R,\Omega}\) stack the real
coordinate vectors of only the observed tensor entries,
using a fixed ordering of \(\Omega\).
Define
\(\mathbf Q_\Omega\in\R^{4|\Omega|\times6}\) by
\begin{equation*}
	\mathbf{Q}_{\Omega}
	=
	\begin{bmatrix}
		\operatorname{vec}_{\R,\Omega}
		(\widehat{\dot{\mathcal X}}^{123}) &
		\operatorname{vec}_{\R,\Omega}
		(\widehat{\dot{\mathcal X}}^{132}) &
		\cdots &
		\operatorname{vec}_{\R,\Omega}
		(\widehat{\dot{\mathcal X}}^{321})
	\end{bmatrix},
\end{equation*}
where the order-specific tensors are defined entrywise by
\eqref{eq:rankR_osqcp_entry}.
With
\(\mathbf y_\Omega
=\operatorname{vec}_{\R,\Omega}(\dot{\mathcal Y})\),
the weight subproblem is
\begin{equation}
	\begin{aligned}
		\min_{\boldsymbol{\omega}\in\R^6}
		\quad
		&
		\frac{1}{2}
		\|\mathbf{Q}_{\Omega}\boldsymbol{\omega}
		-\mathbf{y}_{\Omega}\|_2^2
		+
		\frac{\lambda_{\omega}}{2}
		\|\boldsymbol{\omega}\|_2^2,
		\\
		\mathrm{s.t.}
		\quad
		&
		\boldsymbol{\omega}\geq0,
		\qquad
		\mathbf{1}^T\boldsymbol{\omega}=1.
	\end{aligned}
	\label{eq:omega_update}
\end{equation}
At each outer iteration, this subproblem is updated by a
fixed number \(J_\omega\geq1\) of projected-gradient steps,
initialized with the preceding weight vector.
The precise update is given in
\eqref{eq:weight_projected_gradient}.

For practical implementation, the stopping criterion is based on
the relative change of the observed fitting error. At iteration \(t\), define
\begin{equation*}
	e^{(t)}
	=
	\left\|
	\mathcal{P}_{\Omega}
	\left(
	\widehat{\dot{\mathcal{X}}}^{\mathrm{OE},(t)}
	-
	\dot{\mathcal{Y}}
	\right)
	\right\|_F
	\Big/
	\max\left\{
	\left\|
	\mathcal{P}_{\Omega}(\dot{\mathcal{Y}})
	\right\|_F,
	\epsilon_{\mathrm{mach}}
	\right\},
\end{equation*}
where \(\epsilon_{\mathrm{mach}}>0\) prevents division
by zero. The iteration terminates when
\begin{equation}
	\frac{
		|e^{(t)}-e^{(t-1)}|
	}{
		\max\{e^{(t-1)},\epsilon_{\mathrm{mach}}\}
	}
	<
	\varepsilon,
	\label{eq:stopping_criterion}
\end{equation}
or when the prescribed maximum number of iterations is reached.

The overall procedure is summarized in
Algorithm~\ref{alg:oeqcp_completion}.

\begin{algorithm}[h]
	\caption{Parallel row-wise OE-QCP for quaternion tensor completion}
	\label{alg:oeqcp_completion}
	\begin{algorithmic}[1]
		\STATE \textbf{Input:} Incomplete quaternion tensor
		\(\dot{\mathcal Y}\), observed set \(\Omega\),
		component count \(R\), positive parameters
		\(\lambda_A,\lambda_B,\lambda_C,\lambda_\omega\),
		proximal parameters \(\eta_A,\eta_B,\eta_C\geq0\),
		weight-update count \(J_\omega\), tolerance
		\(\varepsilon\), and maximum iteration count \(T_{\max}\).
		
		\STATE Initialize
		\(\dot{\mathbf A}\), \(\dot{\mathbf B}\),
		\(\dot{\mathbf C}\), and
		\(\boldsymbol{\omega}\in\Delta_6\).
		Set \(t=0\) and compute \(e^{(0)}\).
		
		\REPEAT
		\STATE Set \(t\gets t+1\)
		\STATE Precompute
		\(\mathbf D_A^{\mathrm{OE}}(j,k)\)
		using \eqref{eq:DA_OE}.
		
		\STATE Update all rows of \(\dot{\mathbf A}\)
		in parallel by solving
		\eqref{eq:A_normal_equation_OE}.
		
		\STATE Using the updated \(\dot{\mathbf A}\),
		precompute \(\mathbf D_B^{\mathrm{OE}}(i,k)\).
		
		\STATE Update all rows of \(\dot{\mathbf B}\)
		in parallel by solving
		\eqref{eq:B_normal_equation_OE}.
		
		\STATE Using the updated
		\(\dot{\mathbf A}\) and \(\dot{\mathbf B}\),
		precompute \(\mathbf D_C^{\mathrm{OE}}(i,j)\).
		
		\STATE Update all rows of \(\dot{\mathbf C}\)
		in parallel by solving
		\eqref{eq:C_normal_equation_OE}.
		
		\STATE For adaptive weights, construct
		\(\mathbf Q_\Omega\) and perform \(J_\omega\)
		steps of \eqref{eq:weight_projected_gradient};
		otherwise retain the prescribed weights.
		
		\STATE  Compute \(e^{(t)}\).
		
		\UNTIL{\eqref{eq:stopping_criterion} holds
			or \(t=T_{\max}\).}
		
		\STATE \textbf{Output:} Completed tensor
		\(\dot{\mathcal X}_{\mathrm{rec}}\)
		given by \eqref{eq:completed_tensor_output}.
	\end{algorithmic}
\end{algorithm}

\subsection{Convergence Analysis}

Let
\begin{equation*}
	\begin{aligned}
		\Phi(\dot{\mathbf A},\dot{\mathbf B},
		\dot{\mathbf C},\boldsymbol{\omega})
		={}&
		\frac{1}{2}
		\left\|
		\mathcal{P}_{\Omega}
		\left(
		\widehat{\dot{\mathcal X}}^{\mathrm{OE}}
		-
		\dot{\mathcal Y}
		\right)
		\right\|_F^2
		+
		\frac{\lambda_A}{2}
		\|\dot{\mathbf A}\|_F^2
		+
		\frac{\lambda_B}{2}
		\|\dot{\mathbf B}\|_F^2
		\\
		&+
		\frac{\lambda_C}{2}
		\|\dot{\mathbf C}\|_F^2
		+
		\frac{\lambda_{\omega}}{2}
		\|\boldsymbol{\omega}\|_2^2
	\end{aligned}
\end{equation*}
be the objective function of
\eqref{eq:oeqcp_completion_model}.

At outer iteration \(t\), the weight update uses the newly
updated factors
\(\dot{\mathbf A}^{(t+1)}\),
\(\dot{\mathbf B}^{(t+1)}\), and
\(\dot{\mathbf C}^{(t+1)}\).
Denote the resulting observation matrix by
\(\mathbf Q_\Omega^{(t)}\), and define
\begin{equation*}
	f_t(\boldsymbol{\omega})
	=
	\frac{1}{2}
	\left\|
	\mathbf Q_{\Omega}^{(t)}\boldsymbol{\omega}
	-
	\mathbf y_{\Omega}
	\right\|_2^2
	+
	\frac{\lambda_{\omega}}{2}
	\|\boldsymbol{\omega}\|_2^2.
\end{equation*}
Its gradient is
\(\nabla f_t(\boldsymbol{\omega})
=\mathbf G_t\boldsymbol{\omega}-\mathbf h_t\),
where
\(\mathbf G_t
=(\mathbf Q_\Omega^{(t)})^T\mathbf Q_\Omega^{(t)}
+\lambda_\omega\mathbf I_6\)
and
\(\mathbf h_t
=(\mathbf Q_\Omega^{(t)})^T\mathbf y_\Omega\).
The gradient is Lipschitz continuous with constant
\(L_t=\|\mathbf G_t\|_2>0\).

Starting from
\(\boldsymbol{\omega}^{(t,0)}
=\boldsymbol{\omega}^{(t)}\),
perform
\begin{equation}
	\boldsymbol{\omega}^{(t,\ell+1)}
	=
	\Pi_{\Delta_6}
	\left(
	\boldsymbol{\omega}^{(t,\ell)}
	-
	\frac{1}{L_t}
	\nabla f_t\bigl(\boldsymbol{\omega}^{(t,\ell)}\bigr)
	\right),
	\qquad
	\ell=0,\ldots,J_{\omega}-1,
	\label{eq:weight_projected_gradient}
\end{equation}
and set
\(\boldsymbol{\omega}^{(t+1)}
=\boldsymbol{\omega}^{(t,J_\omega)}\).
Here, \(\Pi_{\Delta_6}\) denotes Euclidean projection
onto the simplex.

\begin{theorem}[Convergence of the OE-QCP completion algorithm]
	\label{thm:oeqcp_convergence}
	Suppose that
	\(\lambda_A,\lambda_B,\lambda_C,\lambda_\omega>0\)
	and that \(\eta_A,\eta_B,\eta_C\geq0\) are fixed.
	Assume that the factor subproblems, including their
	proximal terms when present, are solved exactly in the
	order \(\dot{\mathbf A}\), \(\dot{\mathbf B}\),
	and \(\dot{\mathbf C}\), followed by a fixed positive
	number \(J_\omega\) of updates
	\eqref{eq:weight_projected_gradient}.
	
	If the iterations are continued without a finite stopping
	rule, the generated sequence satisfies:
	\begin{enumerate}
		\item[1)]
		The objective sequence \(\{\Phi^{(t)}\}\)
		is nonincreasing and converges to a finite value.
		
		\item[2)]
		The iterate sequence is bounded, and
		\begin{equation}
			\begin{aligned}
				\lim_{t\to\infty}
				\big(
				&\|\dot{\mathbf A}^{(t+1)}
				-\dot{\mathbf A}^{(t)}\|_F
				+
				\|\dot{\mathbf B}^{(t+1)}
				-\dot{\mathbf B}^{(t)}\|_F
				\\
				&+
				\|\dot{\mathbf C}^{(t+1)}
				-\dot{\mathbf C}^{(t)}\|_F
				+
				\|\boldsymbol{\omega}^{(t+1)}
				-\boldsymbol{\omega}^{(t)}\|_2
				\big)
				=0.
			\end{aligned}
			\label{eq:successive_difference}
		\end{equation}
		
		\item[3)]
		The entire sequence converges to a stationary
		point of \eqref{eq:oeqcp_completion_model}.
	\end{enumerate}
\end{theorem}

The proof of Theorem~\ref{thm:oeqcp_convergence}
is provided in Section~S4 of the supplementary materials. When the weights are fixed, the same argument applies
after removing the weight-update terms and treating
\(\boldsymbol{\omega}\) as a constant.
These results establish convergence to a stationary
point, not necessarily to a global minimizer.

\subsection{Computational Complexity}

Let \(m=|\Omega|\), and let \(R\) be the prescribed
component count. The main cost of one iteration comes
from the row-wise factor updates. Constants associated
with the four real coordinates of a quaternion and the
six multiplication orders are omitted below.

For updating \(\dot{\mathbf A}\), precomputing
\(\mathbf D_A^{\mathrm{OE}}(j,k)\) for all \((j,k)\)
costs \(O(I_2I_3R)\).
Assembling the row-wise normal equations costs
\(O(mR^2)\), and solving the \(I_1\) systems of size
\(4R\times4R\) costs \(O(I_1R^3)\).
Thus, the total cost of this block update is
\(O(I_2I_3R+mR^2+I_1R^3)\).

Similarly, the costs of updating
\(\dot{\mathbf B}\) and \(\dot{\mathbf C}\) are
\(O(I_1I_3R+mR^2+I_2R^3)\) and
\(O(I_1I_2R+mR^2+I_3R^3)\), respectively.
For a fixed number of inner weight updates, the
overall per-iteration cost is therefore
\begin{equation*}
	O\left(
	(I_1I_2+I_1I_3+I_2I_3)R
	+
	3mR^2
	+
	(I_1+I_2+I_3)R^3
	\right).
\end{equation*}

The rows within each factor block can be processed
in parallel. With \(P\) workers, balanced workloads,
and sufficient rows, the row-wise assembly and
linear solves can ideally be accelerated by a factor
of approximately \(P\).
The corresponding per-iteration cost is
\begin{equation*}
	O\left(
	(I_1I_2+I_1I_3+I_2I_3)R
	+
	\frac{
		3mR^2+(I_1+I_2+I_3)R^3
	}{P}
	\right),
\end{equation*}
excluding communication, scheduling, and load-imbalance
overheads. The three factor blocks themselves remain
sequential.

For adaptive weights, constructing the six observed
order-specific reconstructions costs \(O(mR)\).
Forming the six-dimensional quadratic objective costs
\(O(m)\), and \(J_\omega\) projected-gradient steps
cost \(O(J_\omega)\), with dimension-dependent
constants suppressed.
These costs are absorbed into the preceding estimate
when \(J_\omega\) is fixed.
No weight-update cost is incurred when the weights
are prescribed.

\section{Experiments}
\label{sec: experiments}
We first examine order dependence and compare the approximation
performance of OS-QCP and OE-QCP. We then evaluate OE-QCP against
representative baselines on color video inpainting at different
sampling ratios.

All experiments were conducted in MATLAB R2024b on a computer
with an Intel Core Ultra 7 265K processor and 64 GB of RAM.
Video data were normalized to \([0,1]\).

%\subsection{Experimental Setup}

\subsection{Order Sensitivity of Quaternion CP Decomposition}

Order sensitivity is examined using synthetic tensors with
known generating orders and low-rank approximation of
real color videos.

\subsubsection{Synthetic Tensors}

For \((I,R_0)=(20,3)\) and \((25,4)\), generate
\(\dot{\mathbf A},\dot{\mathbf B},\dot{\mathbf C}
\in\Hq^{I\times R_0}\) using a fixed random seed,
with independent standard normal real components
and unit-norm columns.
Each factor triple generates six tensors of size
\(I\times I\times I\), one for each multiplication order:
\begin{equation}
	\dot x_{ijk}^{(\pi)}
	=
	\sum_{r=1}^{R_0}
	\Pi_\pi
	\left(
	\dot a_{ir},\dot b_{jr},\dot c_{kr}
	\right),
	\qquad \pi\in S_3.
	\label{eq:synthetic_order_generation}
\end{equation}
Each tensor is normalized to unit Frobenius norm.

Each target is fitted by the six fixed-order OS-QCP
models and OE-QCP with the same component count \(R=R_0\).
OE-QCP is initialized with uniform weights
\(\omega_\pi=1/6\), which are updated during optimization.
The relative error (RE) is computed as
\(\mathrm{RE}
=\|\widehat{\dot{\mathcal X}}-\dot{\mathcal X}\|_F/
\|\dot{\mathcal X}\|_F\),
where \(\widehat{\dot{\mathcal X}}\) is the tensor
represented by the fitted factors and, for OE-QCP,
the learned weights.

\begin{figure}[h]
	\centering
	\begin{subfigure}[t]{0.49\linewidth}
		\centering
		\includegraphics[width=\linewidth]{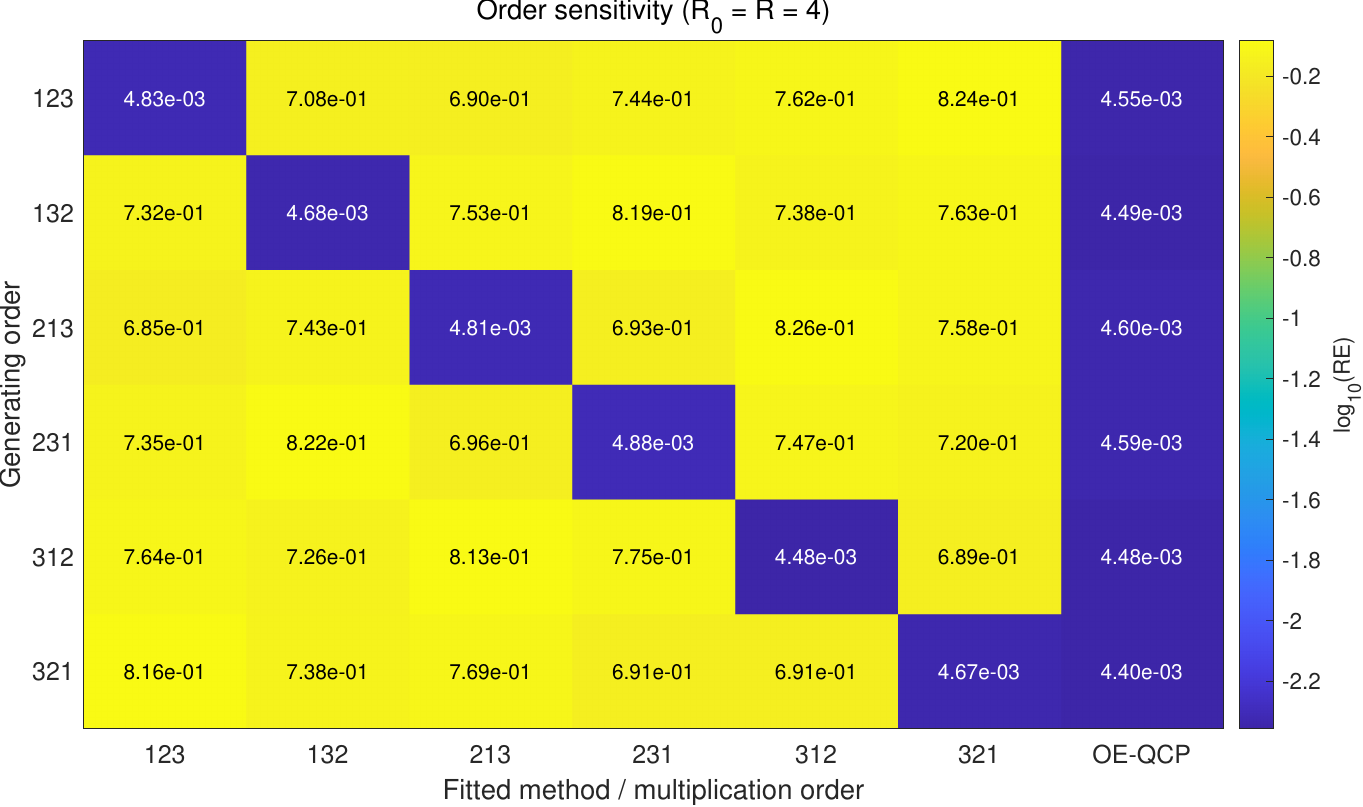}
		\caption{Relative errors: \(I=20\), \(R_0=R=4\).}
		\label{fig:synthetic_re_20}
	\end{subfigure}
	\hfill
	\begin{subfigure}[t]{0.49\linewidth}
		\centering
		\includegraphics[width=\linewidth]{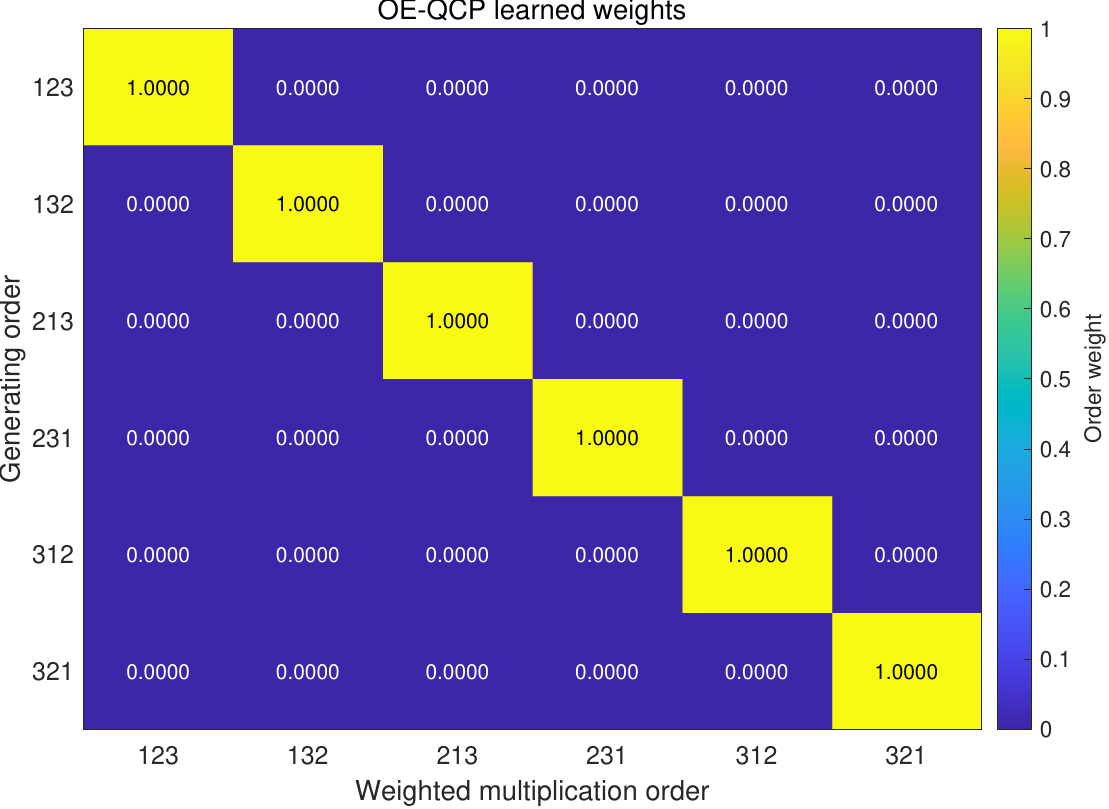}
		\caption{OE-QCP weights: \(I=20\), \(R_0=R=4\).}
		\label{fig:synthetic_weights_20}
	\end{subfigure}
	
	\medskip
	\begin{subfigure}[t]{0.49\linewidth}
		\centering
		\includegraphics[width=\linewidth]{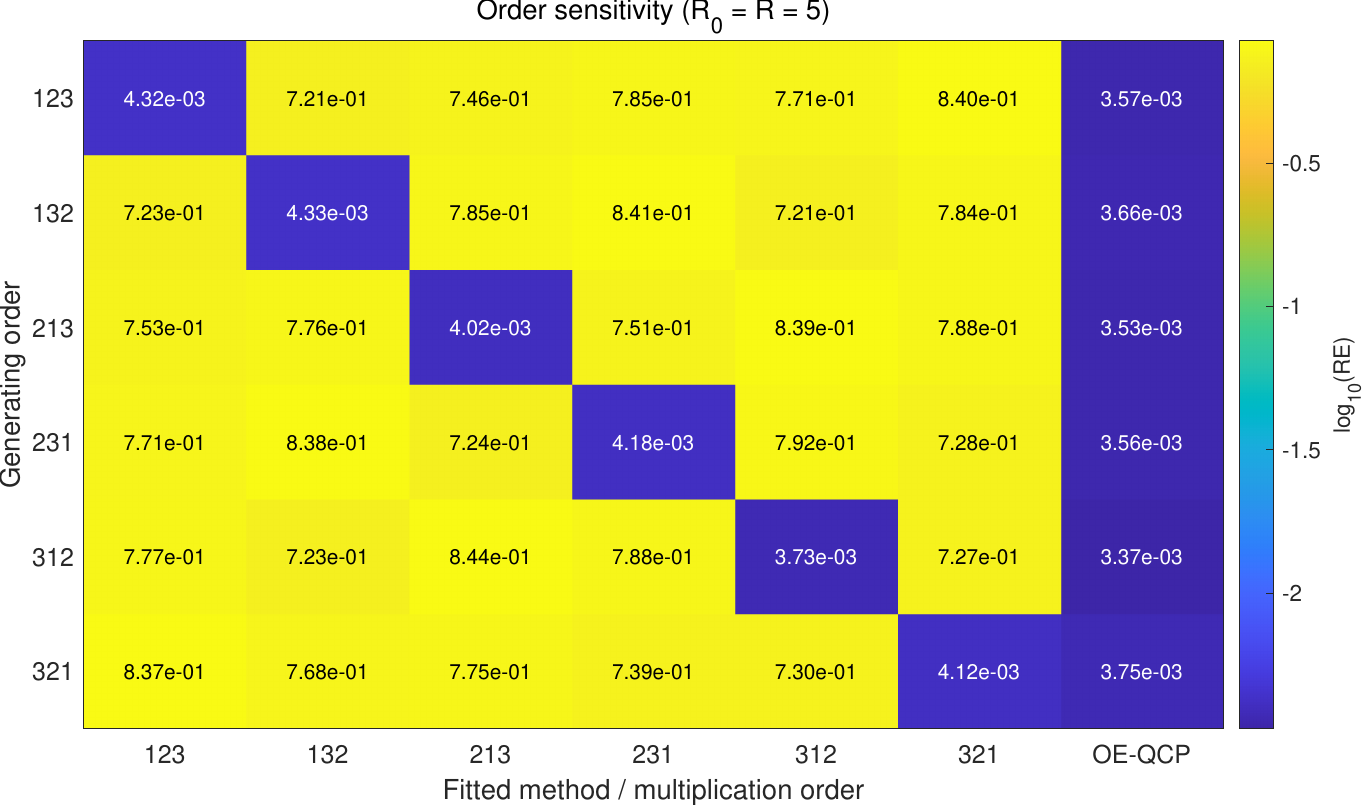}
		\caption{Relative errors: \(I=30\), \(R_0=R=5\).}
		\label{fig:synthetic_re_30}
	\end{subfigure}
	\hfill
	\begin{subfigure}[t]{0.49\linewidth}
		\centering
		\includegraphics[width=\linewidth]{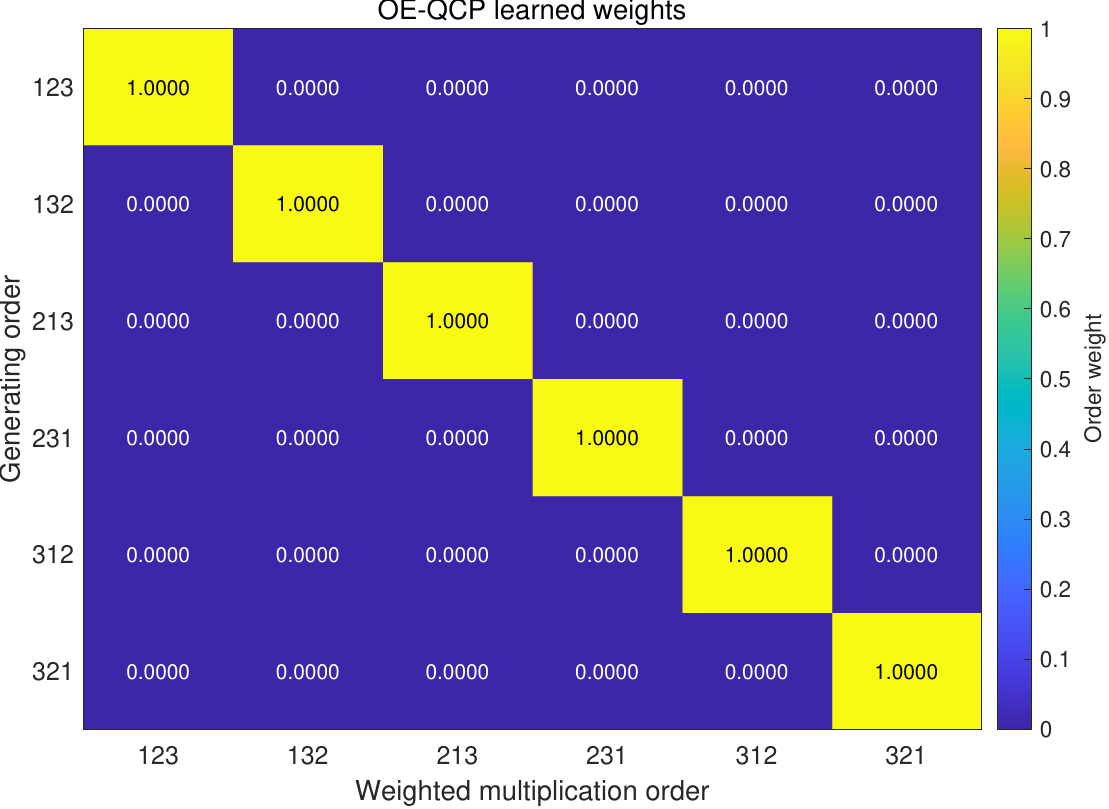}
		\caption{OE-QCP weights: \(I=30\), \(R_0=R=5\).}
		\label{fig:synthetic_weights_30}
	\end{subfigure}
	\caption{Order sensitivity on synthetic tensors of size
		\(I\times I\times I\).
		Rows within each heatmap indicate the generating orders.
		Left: relative errors of the six OS-QCP models and OE-QCP,
		with colors showing \(\log_{10}(\mathrm{RE})\)
		and annotations giving RE.
		Right: order weights learned by OE-QCP from uniform
		initialization.}
	\label{fig:synthetic_order_sensitivity}
\end{figure}

Figure~\ref{fig:synthetic_order_sensitivity} presents
the relative errors and learned weights for the two
settings. In each error heatmap, rows indicate the
generating orders, while columns correspond to the
six OS-QCP fitting orders and OE-QCP.
The accompanying weight heatmap shows the weights
assigned by OE-QCP to the six multiplication orders.

In both settings, OS-QCP achieves substantially lower
errors when the fitting order matches the generating order,
showing that the multiplication order has a substantial
effect on approximation accuracy at a fixed component count.
OE-QCP achieves comparable or slightly lower errors than
the matching-order OS-QCP models, with learned weights
concentrated on the generating orders.
Thus, OE-QCP adapts to the generating order in these
experiments without prior knowledge of that order.

\subsubsection{Low-Rank Approximation of Color Videos}

Experiments are also conducted on the publicly available
RGB video sequences \textit{Hall Monitor} and
\textit{Grandma}, each of size
\(144\times176\times3\times150\)\footnote{Available:
	http://trace.eas.asu.edu/yuv/}.
Each video is represented by a quaternion tensor
\(\dot{\mathcal X}\in\Hq^{I_1\times I_2\times I_3}\),
where \(I_1,I_2\) are the spatial dimensions and \(I_3\)
is the temporal dimension. Each RGB pixel is encoded as a
pure quaternion, with its red, green, and blue intensities
assigned to the three imaginary components.

For these pure quaternion tensors, reversed OS-QCP orders
give equivalent approximation problems by
Remark~\ref{rem:reverse_order}.
We therefore evaluated OS-QCP under orders
\(123\), \(132\), and \(213\), with
\(R\in\{10,\allowbreak 20,\allowbreak 30,\allowbreak
40,\allowbreak 50\}\).
For comparison, we included the conventional real-valued
CP decomposition \cite{harshman1970foundations}, denoted
by R-CP, and the quaternion CP model in
\cite{flamant2026multilinear}, denoted by PQCP.

Approximation performance was evaluated using the peak
signal-to-noise ratio (PSNR), structural similarity index
measure (SSIM), and relative error (RE).
PSNR and SSIM were computed from the reconstructed RGB
videos using the normalized intensity range \([0,1]\).
SSIM was evaluated for each frame and averaged over all
frames. Higher PSNR and SSIM and lower RE indicate better
approximation quality.

\begin{figure*}[htbp]
	\centering
	\begin{tabular}{cc}
		\multicolumn{2}{c}{\textbf{\textit{Hall Monitor}}} \\
		\includegraphics[width=5.5cm]{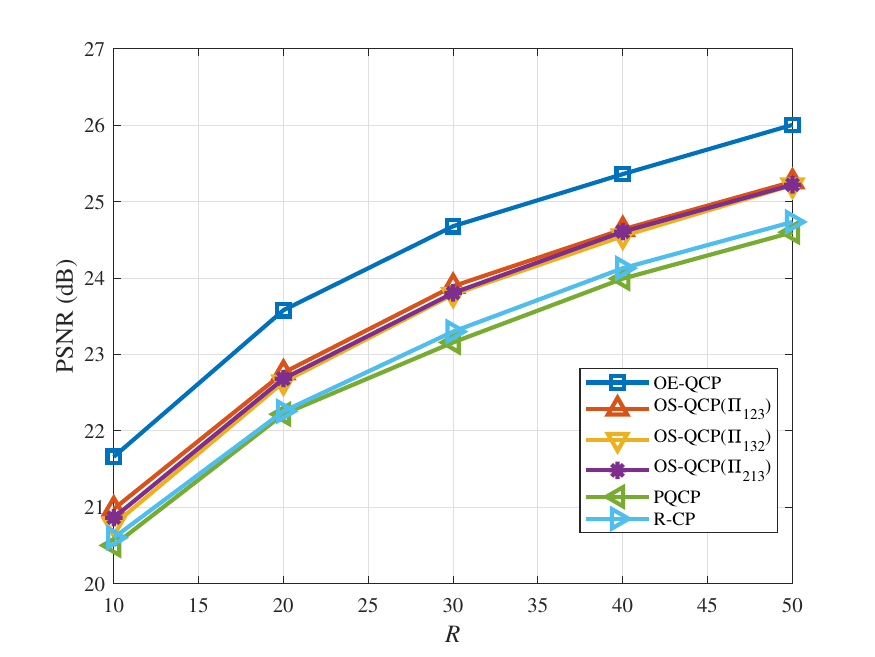} &
		\includegraphics[width=5.5cm]{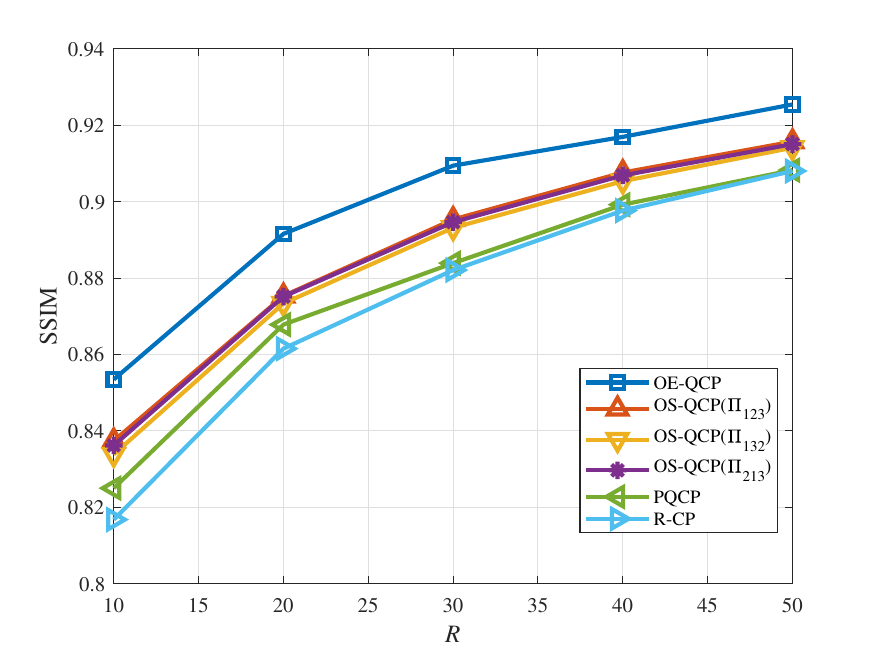} \\
		(a) PSNR & (b) SSIM \\
		[0.3cm]
		\multicolumn{2}{c}{\textbf{\textit{Grandma}}} \\
		\includegraphics[width=5.5cm]{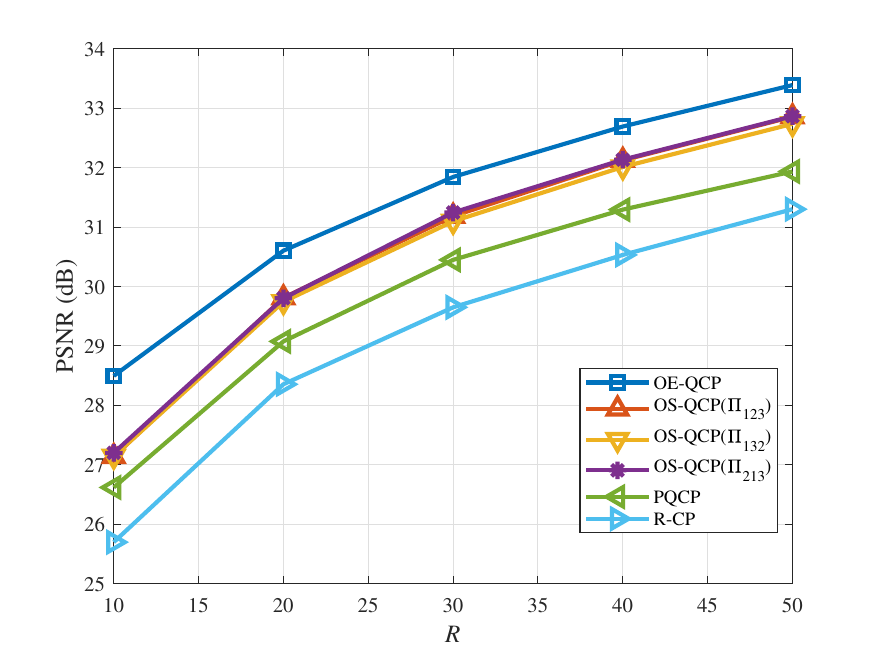} &
		\includegraphics[width=5.5cm]{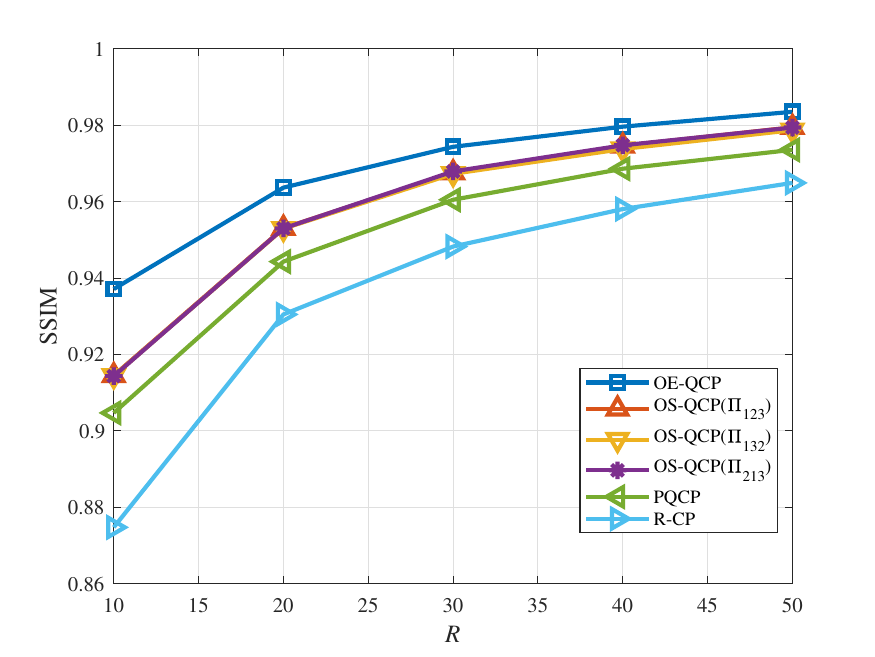} \\
		(c) PSNR  & (d) SSIM 
	\end{tabular}
	\caption{Performance comparison of different decomposition methods
		on the \textit{Hall Monitor} and \textit{Grandma} video sequences
		for different component counts \(R\).}
	\label{fig:rank_performance}
\end{figure*}

Fig.~\ref{fig:rank_performance} presents the quantitative
comparison for different component counts \(R\).
At each prescribed component count, the fixed-order OS-QCP
variants achieve similar approximation accuracy, with modest
variations across multiplication orders.
OE-QCP consistently achieves slightly higher PSNR and SSIM
and lower RE than the fixed-order OS-QCP variants, supporting
the use of weighted ordered products for these video sequences.
More importantly, the learned weights remain close to
\(1/6\) across the tested settings, suggesting that these
natural video data benefit from multiple multiplication
orders rather than favoring a single one.
This further motivates allowing all orders to
participate instead of imposing one fixed order a priori.

In addition, the OS-QCP models also consistently outperform PQCP and R-CP.
Allowing quaternion-valued factors in all three modes provides
greater modeling flexibility than the real-valued middle-factor
restriction in PQCP.
PQCP generally achieves better approximation accuracy than
R-CP on the \textit{Grandma} sequence.
Among the compared methods, OE-QCP achieves the best
approximation accuracy on the tested video sequences.

\begin{figure*}[htbp]
	\centering
	\setlength{\tabcolsep}{1pt}  % 减小列间距，默认是 6pt
	\resizebox{\textwidth}{!}{%
		\begin{tabular}{ccccccc}
			% ------- 顶部大标题 -------
			\multicolumn{7}{c}{\textit{Grandma}} \\
			
			% --- 第一行 (R10) ---
			\includegraphics[width=2cm]{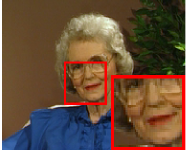} &
			\includegraphics[width=2cm]{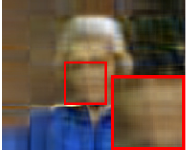} &
			\includegraphics[width=2cm]{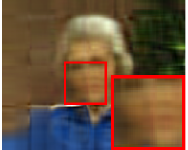} &
			\includegraphics[width=2cm]{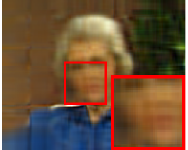} &
			\includegraphics[width=2cm]{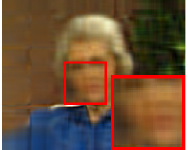} &
			\includegraphics[width=2cm]{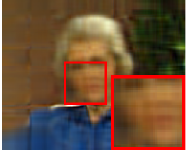} &
			\includegraphics[width=2cm]{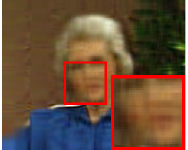} \\
			
			% --- 第二行 (R20) ---
			\includegraphics[width=2cm]{three_CPs/grandma/results/ori_frame_015.pdf} &
			\includegraphics[width=2cm]{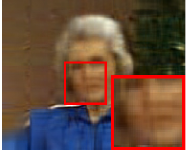} &
			\includegraphics[width=2cm]{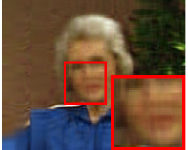} &
			\includegraphics[width=2cm]{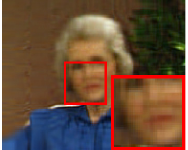} &
			\includegraphics[width=2cm]{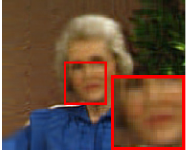} &
			\includegraphics[width=2cm]{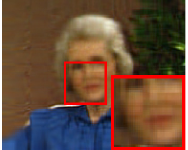} &
			\includegraphics[width=2cm]{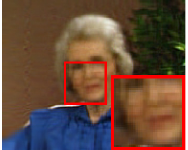} \\
			
			% --- 第三行 (R30) ---
			\includegraphics[width=2cm]{three_CPs/grandma/results/ori_frame_015.pdf} &
			\includegraphics[width=2cm]{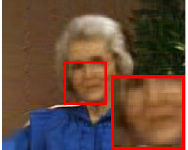} &
			\includegraphics[width=2cm]{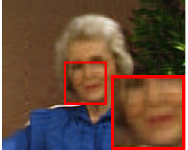} &
			\includegraphics[width=2cm]{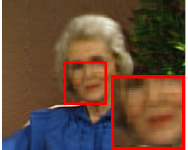} &
			\includegraphics[width=2cm]{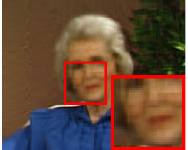} &
			\includegraphics[width=2cm]{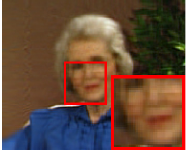} &
			\includegraphics[width=2cm]{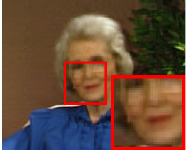} \\
			
			% --- 第四行 (R40) ---
			\includegraphics[width=2cm]{three_CPs/grandma/results/ori_frame_015.pdf} &
			\includegraphics[width=2cm]{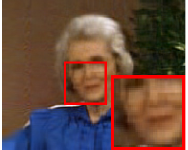} &
			\includegraphics[width=2cm]{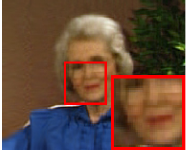} &
			\includegraphics[width=2cm]{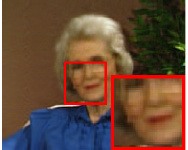} &
			\includegraphics[width=2cm]{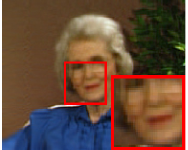} &
			\includegraphics[width=2cm]{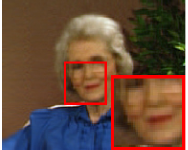} &
			\includegraphics[width=2cm]{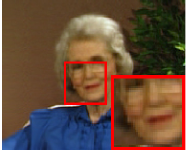} \\
			
			% --- 第五行 (R50) ---
			\includegraphics[width=2cm]{three_CPs/grandma/results/ori_frame_015.pdf} &
			\includegraphics[width=2cm]{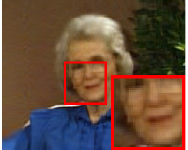} &
			\includegraphics[width=2cm]{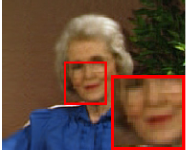} &
			\includegraphics[width=2cm]{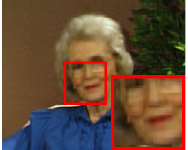} &
			\includegraphics[width=2cm]{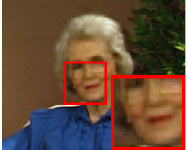} &
			\includegraphics[width=2cm]{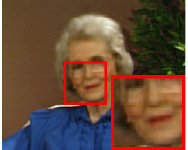} &
			\includegraphics[width=2cm]{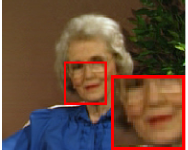} \\

			\small (a) & \small (b) & \small (c) & \small (d) & \small (e) & 
			\small (f) & \small (g)  \\
		\end{tabular}
	}

	\caption{Visual comparison of frame 15 from the \textit{Grandma}
		sequence for different component counts and decomposition methods.
		From top to bottom, the rows correspond to \(R=10,20,30,40,50\).
		From left to right, the columns show the ground truth, R-CP,
		PQCP, OS-QCP under orders \(123\), \(132\), and \(213\),
		and OE-QCP.}
	\label{fig:visual}
\end{figure*}

\begin{figure*}[htbp]
	\centering
	\setlength{\tabcolsep}{1pt}  % 减小列间距，默认是 6pt
	\resizebox{\textwidth}{!}{%
		\begin{tabular}{ccccccc}
			% ------- 顶部大标题 -------
			\multicolumn{7}{c}{\textit{Hall Monitor}} \\
			
			% --- 第一行 (R10) ---
			\includegraphics[width=2cm]{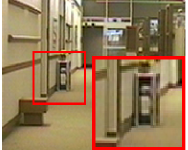} &
			\includegraphics[width=2cm]{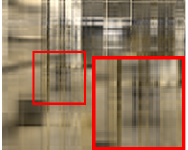} &
			\includegraphics[width=2cm]{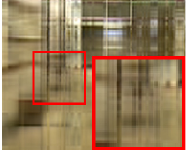} &
			\includegraphics[width=2cm]{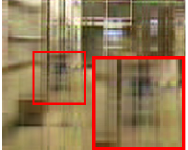} &
			\includegraphics[width=2cm]{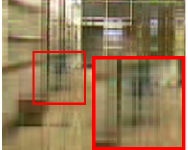} &
			\includegraphics[width=2cm]{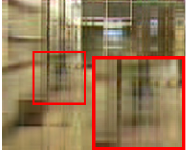} &
			\includegraphics[width=2cm]{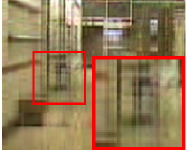} \\
			
			% --- 第二行 (R20) ---
			\includegraphics[width=2cm]{three_CPs/hall_monitor/results/ori_frame_015.pdf} &
			\includegraphics[width=2cm]{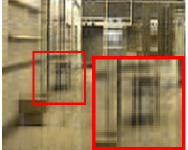} &
			\includegraphics[width=2cm]{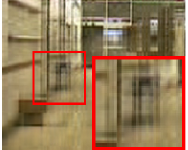} &
			\includegraphics[width=2cm]{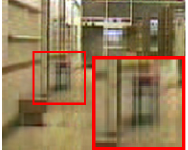} &
			\includegraphics[width=2cm]{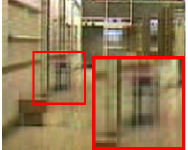} &
			\includegraphics[width=2cm]{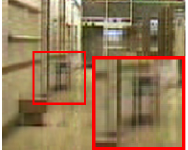} &
			\includegraphics[width=2cm]{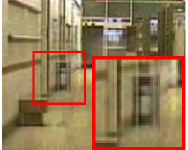} \\
			
			% --- 第三行 (R30) ---
			\includegraphics[width=2cm]{three_CPs/hall_monitor/results/ori_frame_015.pdf} &
			\includegraphics[width=2cm]{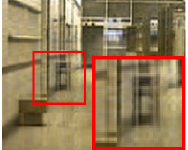} &
			\includegraphics[width=2cm]{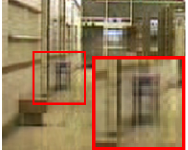} &
			\includegraphics[width=2cm]{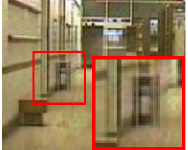} &
			\includegraphics[width=2cm]{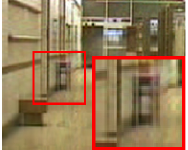} &
			\includegraphics[width=2cm]{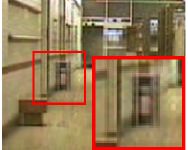} &
			\includegraphics[width=2cm]{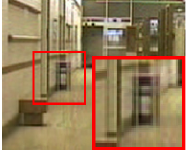} \\
			
			% --- 第四行 (R40) ---
			\includegraphics[width=2cm]{three_CPs/hall_monitor/results/ori_frame_015.pdf} &
			\includegraphics[width=2cm]{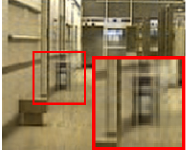} &
			\includegraphics[width=2cm]{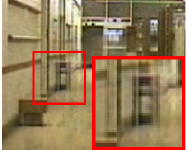} &
			\includegraphics[width=2cm]{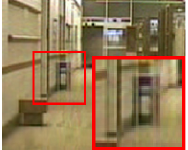} &
			\includegraphics[width=2cm]{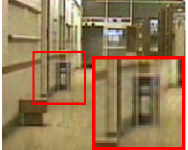} &
			\includegraphics[width=2cm]{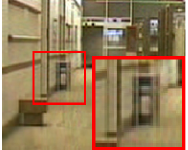} &
			\includegraphics[width=2cm]{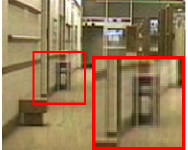} \\
			
			% --- 第五行 (R50) ---
			\includegraphics[width=2cm]{three_CPs/hall_monitor/results/ori_frame_015.pdf} &
			\includegraphics[width=2cm]{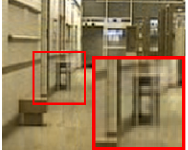} &
			\includegraphics[width=2cm]{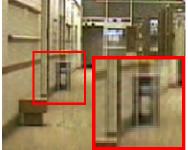} &
			\includegraphics[width=2cm]{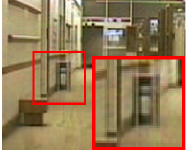} &
			\includegraphics[width=2cm]{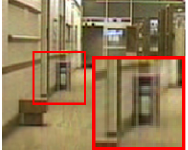} &
			\includegraphics[width=2cm]{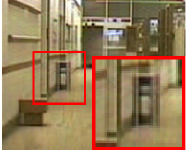} &
			\includegraphics[width=2cm]{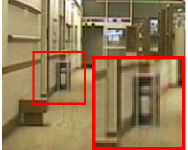} \\
			
			% ------- 底部的子标签行 -------
			\small (a) & \small (b) & \small (c) & \small (d) & \small (e) & 
			\small (f) & \small (g)  \\
		\end{tabular}
	}

	\caption{Visual comparison of frame 15 from the \textit{Hall Monitor}
		sequence for different component counts and decomposition methods.
		From top to bottom, the rows correspond to \(R=10,20,30,40,50\).
		From left to right, the columns show the ground truth, R-CP,
		PQCP, OS-QCP under orders \(123\), \(132\), and \(213\),
		and OE-QCP.}
	\label{fig:visual2}
\end{figure*}

Figs.~\ref{fig:visual} and~\ref{fig:visual2} present representative
reconstructions of the \textit{Grandma} and \textit{Hall Monitor}
sequences, respectively. The visual results agree with the
quantitative comparisons. OE-QCP generally produces the most
faithful reconstructions, particularly at smaller component counts,
whereas R-CP exhibits more pronounced color distortions and artifacts.
For example, the R-CP reconstruction of \textit{Grandma}
in Fig.~\ref{fig:visual} shows noticeable color loss around the mouth.
The visual quality of all methods improves as \(R\) increases,
reflecting the greater modeling capacity provided by additional
components.

Additional results illustrating the evolution of the
relative errors during iteration are provided in Section~S5 of the supplementary materials.

\subsection{Color Video Inpainting Experiments}
In this subsection, we conduct numerical experiments on color video inpainting to evaluate the performance of the proposed OE-QCP model. We compare OE-QCP with several representative baseline methods, including R-CP~\cite{harshman1970foundations}, PQCP~\cite{flamant2026multilinear}, BCPF~\cite{7010937}, LRQTC~\cite{miao2020low}, and CoR-QTURV~\cite{yang2025randomized}. R-CP and BCPF are representative real-valued CP-based methods and provide comparisons with the real-domain counterpart of our OE-QCP. PQCP is included because, to the best of our knowledge, it is the only existing CP model formulated for quaternion tensors. In addition, LRQTC and CoR-QTURV are included as two of the most representative and state-of-the-art quaternion tensor completion methods for color video inpainting.  Eight publicly available color video sequences, namely \textit{Grandma}, \textit{Silent}, \textit{Mother and Daughter}, \textit{Hall Monitor}, \textit{Waterfall}, \textit{Suzie}, \textit{News}, and \textit{Claire}, are selected for the experiments. For each video sequence, the first 150 frames are used, and each video is represented as a third-order quaternion tensor $\dot{\mathcal{X}}\in\mathbb{H}^{144\times176\times150}$.

The reconstruction quality is evaluated using PSNR and SSIM. The sampling ratio (SR) is defined as $
\mathrm{SR}
=
\frac{|\Omega|}{I_1 I_2 I_3}$,
where $\Omega\in\{0,1\}^{I_1\times I_2\times I_3}$ denotes the randomly generated observation mask. The same mask is applied to all four components of the quaternion-valued tensor. The regularization parameters of OE-QCP are set to $\lambda_A=\lambda_B=\lambda_C=10^{-3}$ and $\lambda_{\omega}=10^{-8}$, while the proximal parameters are set to $\eta_A=\eta_B=\eta_C=10^{-6}$. For $\mathrm{SR}\in\{0.1,0.2,0.3\}$, the component counts
are empirically set to $R=55$, $75$, and $95$, respectively.

\begin{table*}[htbp]
	\caption{A comparison of quantitative assessment indices (PSNR/SSIM) across different methods on eight color video sequences.}
	\label{table_comparison}
	\centering
	\resizebox{\textwidth}{!}{
		\begin{tabular}{|c|c|c|c|c|c|c|}
			\hline
			Methods: & BCPF & LRQTC & R-CP & PQCP & CoR-QTURV & \textbf{OE-QCP} \\
			\hline
			Color Videos: & \multicolumn{6}{c|}{SR = 0.3} \\
			\hline
			\textit{Grandma} & 34.0253/0.9847 & 30.1204/0.9603 & 34.2809/0.9834 & 34.8335/0.9878 & 33.5118/0.9835 & \textbf{36.7675}/\textbf{0.9923} \\
			\textit{Silent} & 27.9257/0.8500 & 28.9764/0.8832 & 28.8168/0.8756 & 28.8538/0.8765 & 28.3576/0.8871 & \textbf{30.6294}/\textbf{0.9145} \\
			\textit{Mother and Daughter} & 31.9827/0.9793 & 34.2278/0.9873 & 33.0798/0.9832 & 33.1803/0.9838 & 31.9520/0.9808 & \textbf{34.6029}/\textbf{0.9875} \\
			\textit{Hall Monitor} & 26.4554/0.9276 & 26.1402/0.9279 & 27.2849/0.9401 & 27.0214/0.9383 & 27.8232/0.9483 & \textbf{29.1212}/\textbf{0.9570} \\
			\textit{Waterfall} & 26.8920/0.8639 & 27.1696/0.8862 & 27.4027/0.8863 & 27.5874/0.8915 & \textbf{29.2809}/\textbf{0.9209} & 28.8735/0.9156 \\
			\textit{Suzie} & 28.8761/0.8979 & 30.0886/\textbf{0.9391} & 29.6794/0.9129 & 29.4982/0.9126 & 29.5483/0.9223 & \textbf{31.1754}/0.9371 \\
			\textit{News} & 25.2023/0.8790 & 26.5092/0.9066 & 25.7103/0.8869 & 25.7918/0.8927 & 26.7405/0.9149 & \textbf{27.5363}/\textbf{0.9242} \\
			\textit{Claire} & 34.5451/0.9890 & 31.1959/0.9805 & 34.8901/0.9893 & 34.9223/0.9901 & 33.9144/0.9889 & \textbf{37.5227}/\textbf{0.9944} \\
			\hline
			Aver. & 29.4881/0.9214 & 29.3035/0.9339 & 30.1431/0.9322 & 30.2111/0.9342 & 30.1411/0.9433 & \textbf{32.0286}/\textbf{0.9528} \\
			\hline
			Color Videos: & \multicolumn{6}{c|}{SR = 0.2} \\
			\hline
			\textit{Grandma} & 32.7891/0.9790 & 27.6464/0.9216 & 32.7661/0.9760 & 33.2829/0.9824 & 31.7123/0.9751 & \textbf{35.4427}/\textbf{0.9897} \\
			\textit{Silent} & 27.1054/0.8282 & 27.1728/0.8184 & 27.5835/0.8461 & 27.6308/0.8472 & 26.6212/0.8473 & \textbf{29.4282}/\textbf{0.8917} \\
			\textit{Mother and Daughter} & 31.1867/0.9766 & 31.6512/0.9800 & 31.7538/0.9784 & 32.0278/0.9799 & 30.4640/0.9754 & \textbf{33.5480}/\textbf{0.9847} \\
			\textit{Hall Monitor} & 25.2801/0.9131 & 24.0645/0.8962 & 25.7793/0.9211 & 25.7136/0.9230 & 26.1712/0.9331 & \textbf{27.8728}/\textbf{0.9458} \\
			\textit{Waterfall} & 26.0331/0.8432 & 25.6241/0.8473 & 26.2136/0.8565 & 26.4293/0.8640 & 27.4252/0.8853 & \textbf{27.8055}/\textbf{0.8926} \\
			\textit{Suzie} & 27.7907/0.8809 & 27.4123/0.8825 & 28.2714/0.8924 & 28.1662/0.8920 & 27.6961/0.8983 & \textbf{30.0168}/\textbf{0.9224} \\
			\textit{News} & 24.3254/0.8592 & 24.5511/0.8709 & 24.3791/0.8530 & 24.4929/0.8603 & 24.9607/0.8851 & \textbf{26.2490}/\textbf{0.9031} \\
			\textit{Claire} & 33.2878/0.9860 & 28.4367/0.9656 & 33.0540/0.9847 & 33.1636/0.9858 & 31.7361/0.9829 & \textbf{35.8152}/\textbf{0.9920} \\
			\hline
			Aver. & 28.4748/0.9083 & 27.0699/0.8978 & 28.7251/0.9135 & 28.8634/0.9168 & 28.3484/0.9228 & \textbf{30.7723}/\textbf{0.9403} \\
			\hline
			Color Videos: & \multicolumn{6}{c|}{SR = 0.1} \\
			\hline
			\textit{Grandma} & 30.9795/0.9664 & 25.0058/0.8879 & 30.1954/0.9572 & 31.2080/0.9718 & 29.3056/0.9535 & \textbf{33.7344}/\textbf{0.9842} \\
			\textit{Silent} & 25.3119/0.7774 & 24.6607/0.7487 & 25.7676/0.7875 & 26.0258/0.8061 & 24.3104/0.7821 & \textbf{28.0788}/\textbf{0.8587} \\
			\textit{Mother and Daughter} & 29.5770/0.9699 & 28.3985/0.9667 & 30.1687/0.9714 & 30.4311/0.9738 & 28.5922/0.9673 & \textbf{32.2164}/\textbf{0.9803} \\
			\textit{Hall Monitor} & 23.2315/0.8801 & 20.7005/0.8250 & 23.7843/0.8896 & 23.8078/0.8977 & 23.6013/0.9012 & \textbf{26.3116}/\textbf{0.9287} \\
			\textit{Waterfall} & 24.6446/0.8058 & 23.2744/0.7819 & 24.6375/0.8071 & 24.9918/0.8212 & 25.1322/0.8261 & \textbf{26.7262}/\textbf{0.8646} \\
			\textit{Suzie} & 25.7145/0.8466 & 24.6764/0.8471 & 26.2007/0.8481 & 26.3111/0.8573 & 24.8799/0.8504 & \textbf{28.4764}/\textbf{0.8991} \\
			\textit{News} & 22.3569/0.7997 & 21.1162/0.7858 & 22.3487/0.7841 & 22.6350/0.8116 & 22.6191/0.8268 & \textbf{24.8490}/\textbf{0.8724} \\
			\textit{Claire} & 31.0433/0.9778 & 24.8183/0.9227 & 29.2869/0.9643 & 30.4167/0.9744 & 29.2043/0.9714 & \textbf{33.7826}/\textbf{0.9882} \\
			\hline
			Aver. & 26.6074/0.8780 & 24.0814/0.8457 & 26.5487/0.8762 & 26.9784/0.8892 & 25.9556/0.8849 & \textbf{29.2719}/\textbf{0.9220} \\
			\hline
		\end{tabular}
	}
\end{table*}

Table~\ref{table_comparison} shows that OE-QCP achieves the
highest average PSNR and SSIM over the eight video sequences
at all three sampling ratios, with average PSNR gains of
$1.82$, $1.91$, and $2.29$ dB over PQCP, respectively,
although some baselines perform better on individual sequences.

\begin{figure*}[htbp]
	\centering
	\setlength{\tabcolsep}{1pt} 
	\resizebox{\textwidth}{!}{%
		\begin{tabular}{cccccccc}
			%			\multicolumn{7}{c}{} \\
			
			% ---  ---
			\includegraphics[width=2cm]{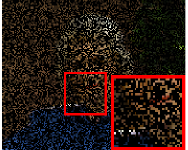} &
			\includegraphics[width=2cm]{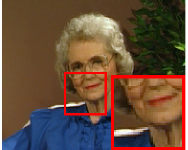} &
			\includegraphics[width=2cm]{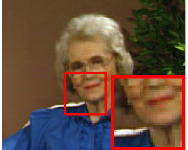} &
			\includegraphics[width=2cm]{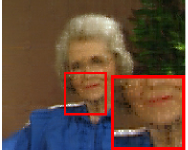} &
			\includegraphics[width=2cm]{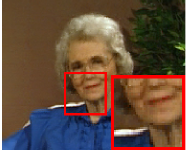} &
			\includegraphics[width=2cm]{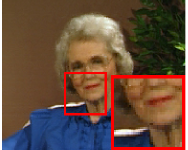} &
			\includegraphics[width=2cm]{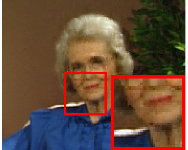} &
			\includegraphics[width=2cm]{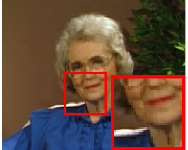} \\
			
			% ---  ---
			\includegraphics[width=2cm]{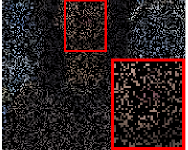} &
			\includegraphics[width=2cm]{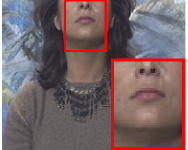} &
			\includegraphics[width=2cm]{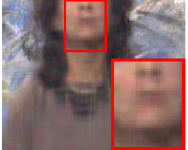} &
			\includegraphics[width=2cm]{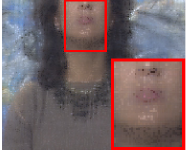} &
			\includegraphics[width=2cm]{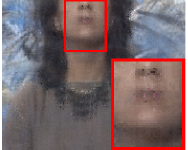} &
			\includegraphics[width=2cm]{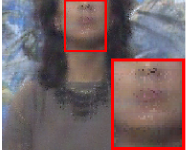} &
			\includegraphics[width=2cm]{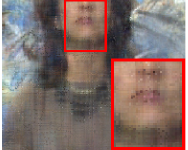} &
			\includegraphics[width=2cm]{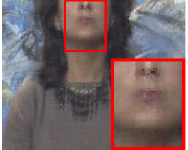} \\
			
			% --- ---
			\includegraphics[width=2cm]{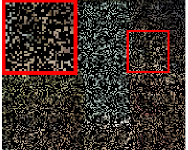} &
			\includegraphics[width=2cm]{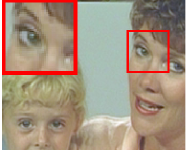} &
			\includegraphics[width=2cm]{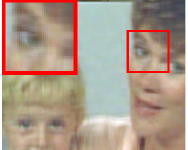} &
			\includegraphics[width=2cm]{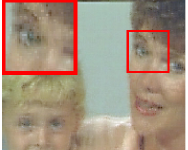} &
			\includegraphics[width=2cm]{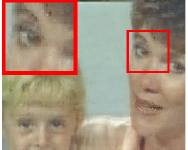} &
			\includegraphics[width=2cm]{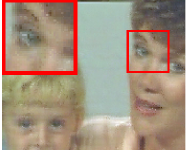} &
			\includegraphics[width=2cm]{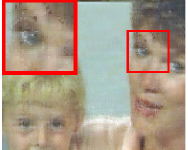} &
			\includegraphics[width=2cm]{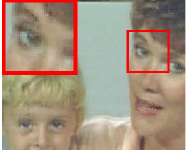} \\
			
			% ---  ---
			\includegraphics[width=2cm]{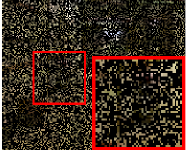} &
			\includegraphics[width=2cm]{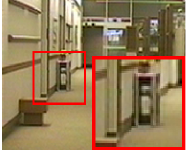} &
			\includegraphics[width=2cm]{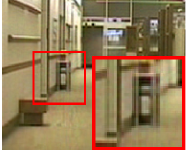} &
			\includegraphics[width=2cm]{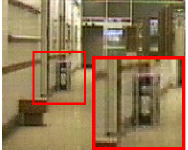} &
			\includegraphics[width=2cm]{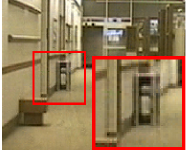} &
			\includegraphics[width=2cm]{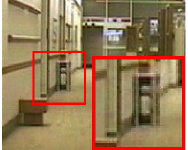} &
			\includegraphics[width=2cm]{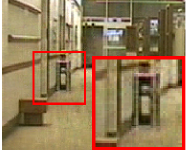} &
			\includegraphics[width=2cm]{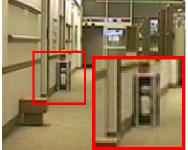} \\
			
			% --- ---
			\includegraphics[width=2cm]{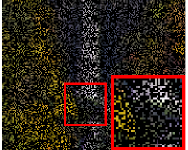} &
			\includegraphics[width=2cm]{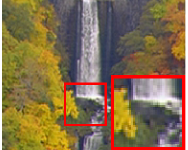} &
			\includegraphics[width=2cm]{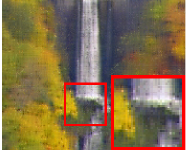} &
			\includegraphics[width=2cm]{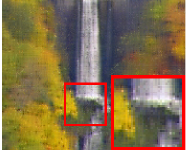} &
			\includegraphics[width=2cm]{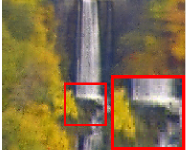} &
			\includegraphics[width=2cm]{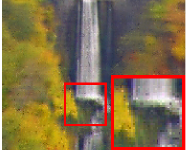} &
			\includegraphics[width=2cm]{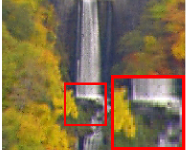} &
			\includegraphics[width=2cm]{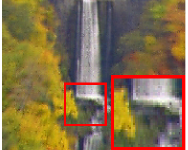} \\
			
			\includegraphics[width=2cm]{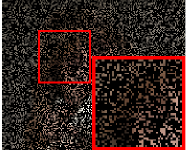} &
			\includegraphics[width=2cm]{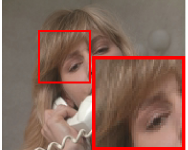} &
			\includegraphics[width=2cm]{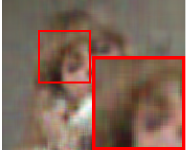} &
			\includegraphics[width=2cm]{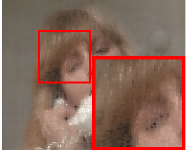} &
			\includegraphics[width=2cm]{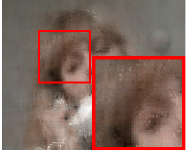} &
			\includegraphics[width=2cm]{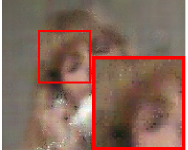} &
			\includegraphics[width=2cm]{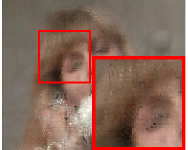} &
			\includegraphics[width=2cm]{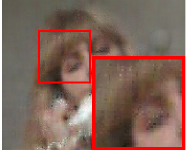} \\
			
			\includegraphics[width=2cm]{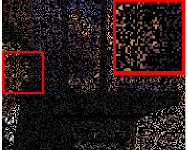} &
			\includegraphics[width=2cm]{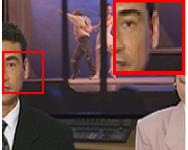} &
			\includegraphics[width=2cm]{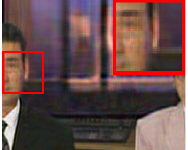} &
			\includegraphics[width=2cm]{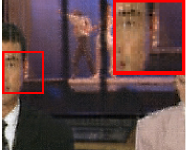} &
			\includegraphics[width=2cm]{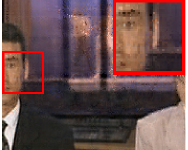} &
			\includegraphics[width=2cm]{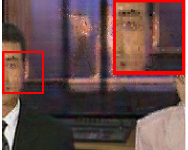} &
			\includegraphics[width=2cm]{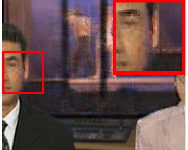} &
			\includegraphics[width=2cm]{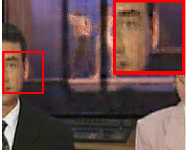} \\
			
			\includegraphics[width=2cm]{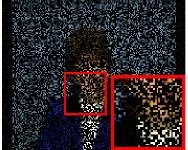} &
			\includegraphics[width=2cm]{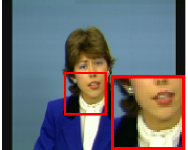} &
			\includegraphics[width=2cm]{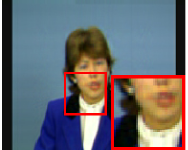} &
			\includegraphics[width=2cm]{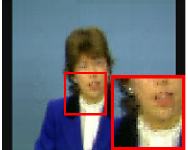} &
			\includegraphics[width=2cm]{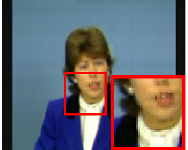} &
			\includegraphics[width=2cm]{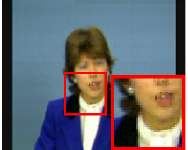} &
			\includegraphics[width=2cm]{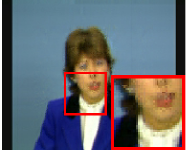} &
			\includegraphics[width=2cm]{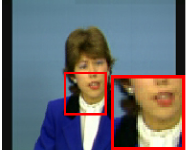} \\
			\small (a) & \small (b) & \small (c) & \small (d) & \small (e) & 
			\small (f) & \small (g)& \small (h)  \\
		\end{tabular}
	}

	\caption{Video completion results at $\mathrm{SR}=0.3$.
		Each row shows a representative frame from a different video sequence.
		Columns from left to right: observed, ground truth, BCPF,
		LRQTC, R-CP, PQCP, CoR-QTURV, and OE-QCP.}
	\label{fig:completion03}
\end{figure*}

\begin{figure*}[htbp]
	\centering
	\setlength{\tabcolsep}{1pt}  % 
	\resizebox{\textwidth}{!}{%
		\begin{tabular}{cccccccc}
			%			\multicolumn{7}{c}{} \\
			
			% --- ---
			\includegraphics[width=2cm]{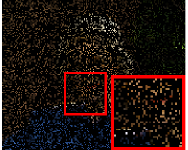} &
			\includegraphics[width=2cm]{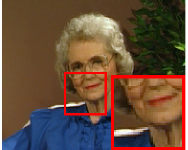} &
			\includegraphics[width=2cm]{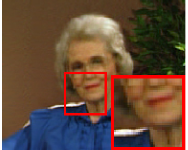} &
			\includegraphics[width=2cm]{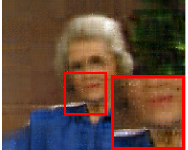} &
			\includegraphics[width=2cm]{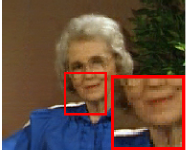} &
			\includegraphics[width=2cm]{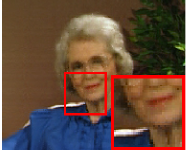} &
			\includegraphics[width=2cm]{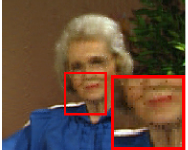} &
			\includegraphics[width=2cm]{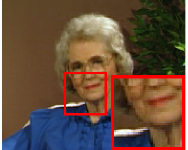} \\
			
			% --- ---
			\includegraphics[width=2cm]{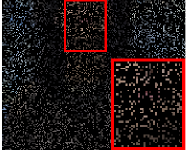} &
			\includegraphics[width=2cm]{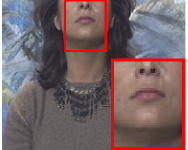} &
			\includegraphics[width=2cm]{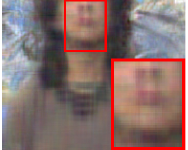} &
			\includegraphics[width=2cm]{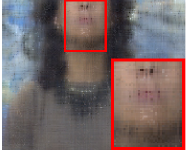} &
			\includegraphics[width=2cm]{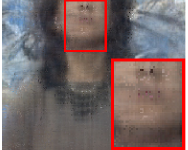} &
			\includegraphics[width=2cm]{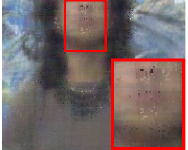} &
			\includegraphics[width=2cm]{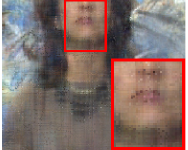} &
			\includegraphics[width=2cm]{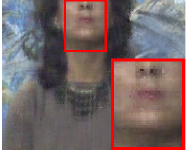} \\
			
			% --- ---
			\includegraphics[width=2cm]{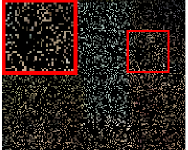} &
			\includegraphics[width=2cm]{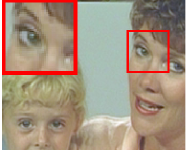} &
			\includegraphics[width=2cm]{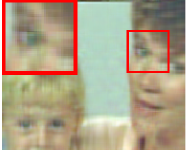} &
			\includegraphics[width=2cm]{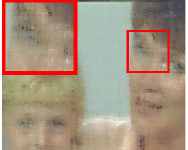} &
			\includegraphics[width=2cm]{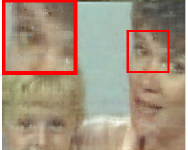} &
			\includegraphics[width=2cm]{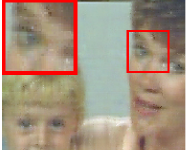} &
			\includegraphics[width=2cm]{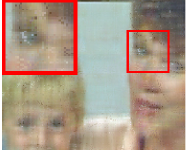} &
			\includegraphics[width=2cm]{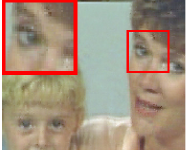} \\
			
			% --- ---
			\includegraphics[width=2cm]{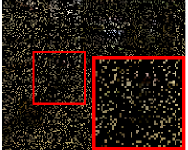} &
			\includegraphics[width=2cm]{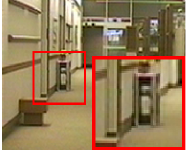} &
			\includegraphics[width=2cm]{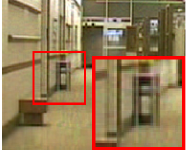} &
			\includegraphics[width=2cm]{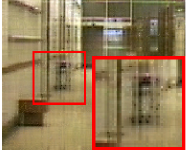} &
			\includegraphics[width=2cm]{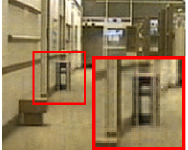} &
			\includegraphics[width=2cm]{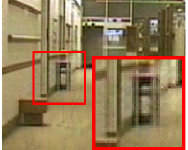} &
			\includegraphics[width=2cm]{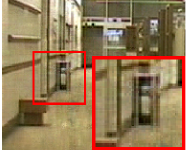} &
			\includegraphics[width=2cm]{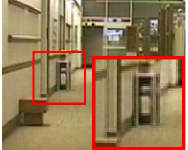} \\
			
			% --- ---
			\includegraphics[width=2cm]{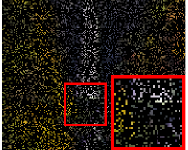} &
			\includegraphics[width=2cm]{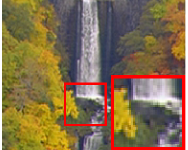} &
			\includegraphics[width=2cm]{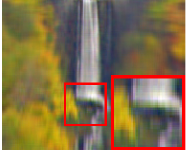} &
			\includegraphics[width=2cm]{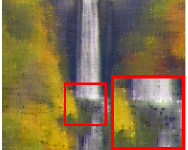} &
			\includegraphics[width=2cm]{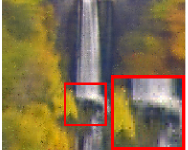} &
			\includegraphics[width=2cm]{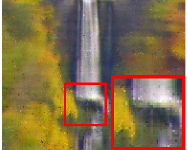} &
			\includegraphics[width=2cm]{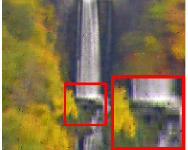} &
			\includegraphics[width=2cm]{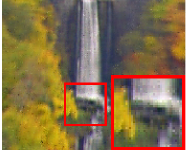} \\
			
			\includegraphics[width=2cm]{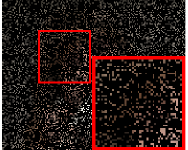} &
			\includegraphics[width=2cm]{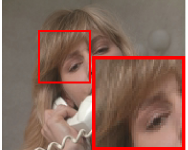} &
			\includegraphics[width=2cm]{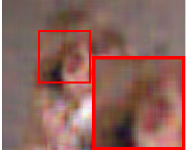} &
			\includegraphics[width=2cm]{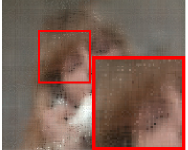} &
			\includegraphics[width=2cm]{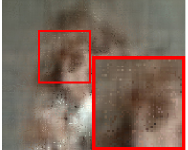} &
			\includegraphics[width=2cm]{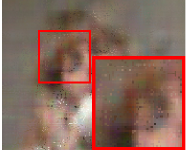} &
			\includegraphics[width=2cm]{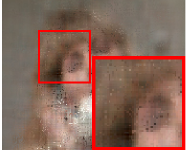} &
			\includegraphics[width=2cm]{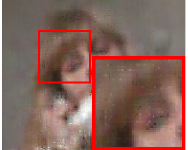} \\
			
			\includegraphics[width=2cm]{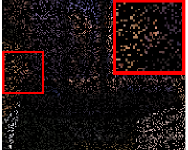} &
			\includegraphics[width=2cm]{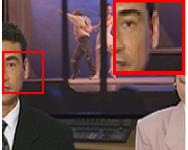} &
			\includegraphics[width=2cm]{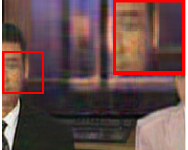} &
			\includegraphics[width=2cm]{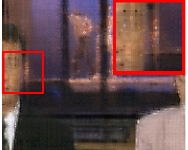} &
			\includegraphics[width=2cm]{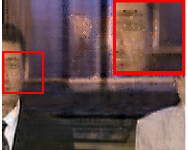} &
			\includegraphics[width=2cm]{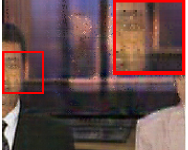} &
			\includegraphics[width=2cm]{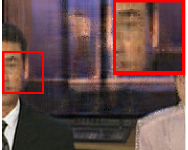} &
			\includegraphics[width=2cm]{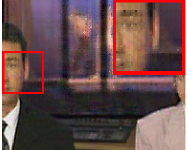} \\
			
			\includegraphics[width=2cm]{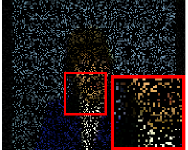} &
			\includegraphics[width=2cm]{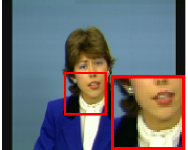} &
			\includegraphics[width=2cm]{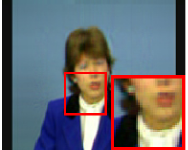} &
			\includegraphics[width=2cm]{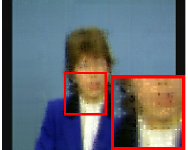} &
			\includegraphics[width=2cm]{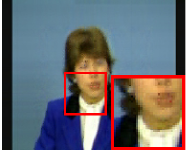} &
			\includegraphics[width=2cm]{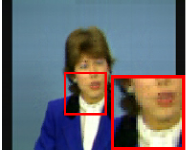} &
			\includegraphics[width=2cm]{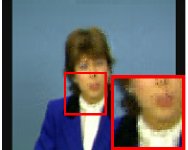} &
			\includegraphics[width=2cm]{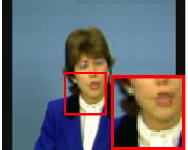} \\
			\small (a) & \small (b) & \small (c) & \small (d) & \small (e) & 
			\small (f) & \small (g) & \small (h)  \\
		\end{tabular}
	}

	\caption{Video completion results at $\mathrm{SR}=0.2$.
		Each row shows a representative frame from a different video sequence.
		Columns from left to right: observed, ground truth, BCPF,
		LRQTC, R-CP, PQCP, CoR-QTURV, and OE-QCP.}
	\label{fig:completion02}
\end{figure*}

\begin{figure*}[htbp]
	\centering
	\setlength{\tabcolsep}{1pt}  % 
	\resizebox{\textwidth}{!}{%
		\begin{tabular}{cccccccc}
			%			\multicolumn{7}{c}{} \\
			
			% --- ---
			\includegraphics[width=2cm]{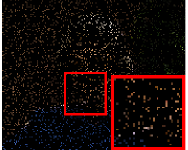} &
			\includegraphics[width=2cm]{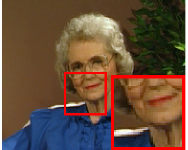} &
			\includegraphics[width=2cm]{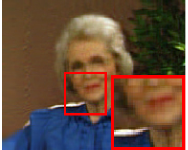} &
			\includegraphics[width=2cm]{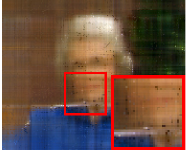} &
			\includegraphics[width=2cm]{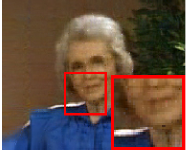} &
			\includegraphics[width=2cm]{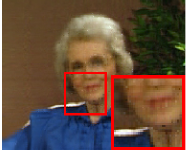} &
			\includegraphics[width=2cm]{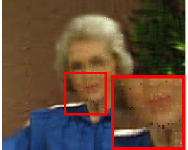} &
			\includegraphics[width=2cm]{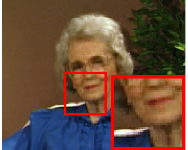} \\
			
			% ---  ---
			\includegraphics[width=2cm]{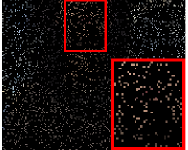} &
			\includegraphics[width=2cm]{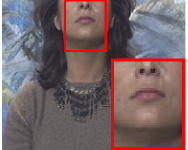} &
			\includegraphics[width=2cm]{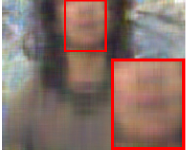} &
			\includegraphics[width=2cm]{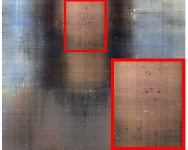} &
			\includegraphics[width=2cm]{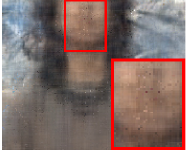} &
			\includegraphics[width=2cm]{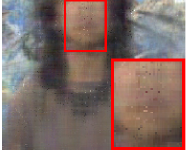} &
			\includegraphics[width=2cm]{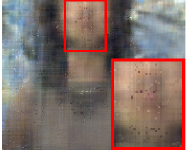} &
			\includegraphics[width=2cm]{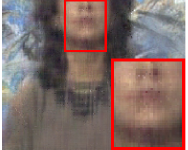} \\
			
			% ---  ---
			\includegraphics[width=2cm]{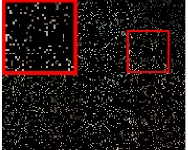} &
			\includegraphics[width=2cm]{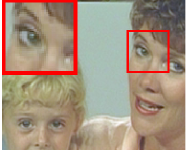} &
			\includegraphics[width=2cm]{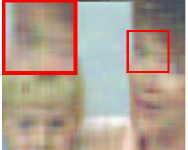} &
			\includegraphics[width=2cm]{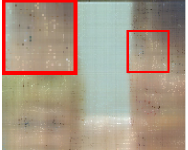} &
			\includegraphics[width=2cm]{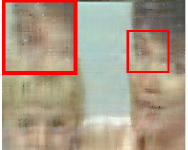} &
			\includegraphics[width=2cm]{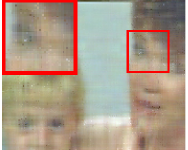} &
			\includegraphics[width=2cm]{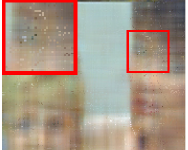} &
			\includegraphics[width=2cm]{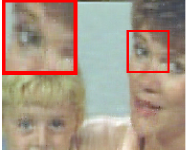} \\
			
			% --- ---
			\includegraphics[width=2cm]{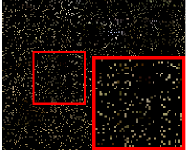} &
			\includegraphics[width=2cm]{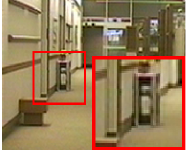} &
			\includegraphics[width=2cm]{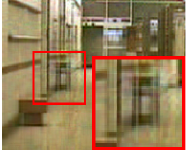} &
			\includegraphics[width=2cm]{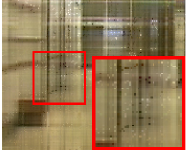} &
			\includegraphics[width=2cm]{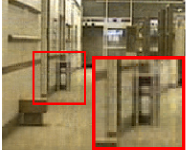} &
			\includegraphics[width=2cm]{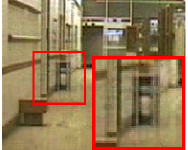} &
			\includegraphics[width=2cm]{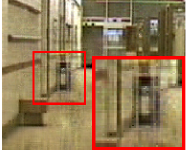} &
			\includegraphics[width=2cm]{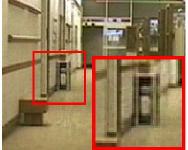} \\
			
			% ------
			\includegraphics[width=2cm]{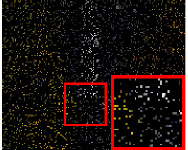} &
			\includegraphics[width=2cm]{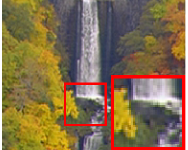} &
			\includegraphics[width=2cm]{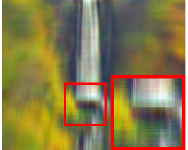} &
			\includegraphics[width=2cm]{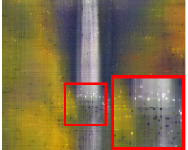} &
			\includegraphics[width=2cm]{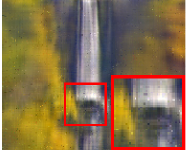} &
			\includegraphics[width=2cm]{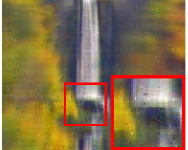} &
			\includegraphics[width=2cm]{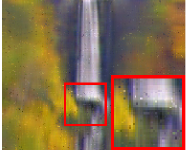} &
			\includegraphics[width=2cm]{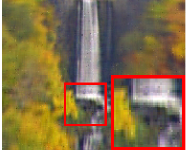} \\
			
			\includegraphics[width=2cm]{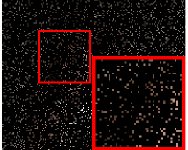} &
			\includegraphics[width=2cm]{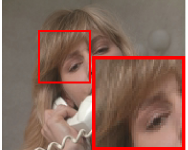} &
			\includegraphics[width=2cm]{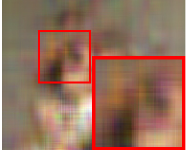} &
			\includegraphics[width=2cm]{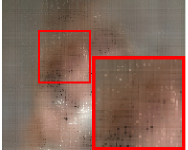} &
			\includegraphics[width=2cm]{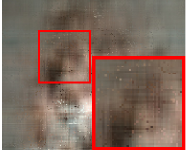} &
			\includegraphics[width=2cm]{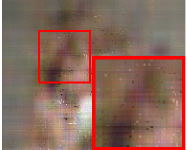} &
			\includegraphics[width=2cm]{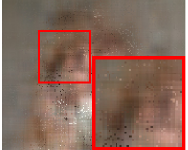} &
			\includegraphics[width=2cm]{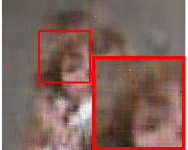} \\
			
			\includegraphics[width=2cm]{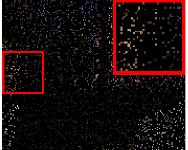} &
			\includegraphics[width=2cm]{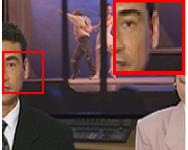} &
			\includegraphics[width=2cm]{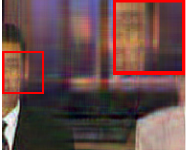} &
			\includegraphics[width=2cm]{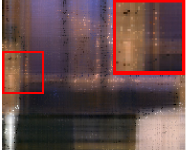} &
			\includegraphics[width=2cm]{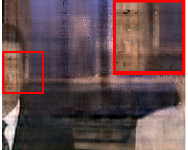} &
			\includegraphics[width=2cm]{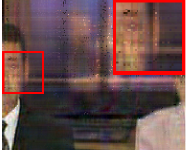} &
			\includegraphics[width=2cm]{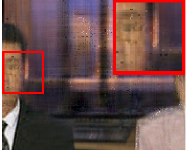} &
			\includegraphics[width=2cm]{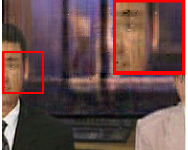} \\
			
			\includegraphics[width=2cm]{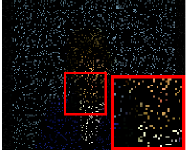} &
			\includegraphics[width=2cm]{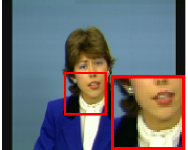} &
			\includegraphics[width=2cm]{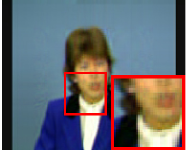} &
			\includegraphics[width=2cm]{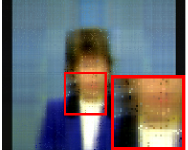} &
			\includegraphics[width=2cm]{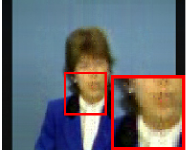} &
			\includegraphics[width=2cm]{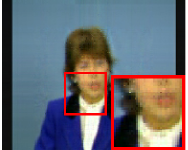} &
			\includegraphics[width=2cm]{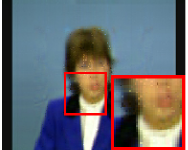} &
			\includegraphics[width=2cm]{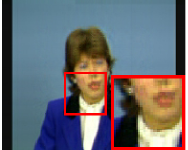} \\
			\small (a) & \small (b) & \small (c) & \small (d) & \small (e) & 
			\small (f) & \small (g) & \small (h)  \\
		\end{tabular}
	}

	\caption{Video completion results at $\mathrm{SR}=0.1$.
		Each row shows a representative frame from a different video sequence.
		Columns from left to right: observed, ground truth, BCPF,
		LRQTC, R-CP, PQCP, CoR-QTURV, and OE-QCP.}
	\label{fig:completion01}
\end{figure*}

Figs.~\ref{fig:completion03}--\ref{fig:completion01} compare
the visual results at $\mathrm{SR}=0.3$, $0.2$, and $0.1$,
respectively. Most methods recover the main structures and
colors at $\mathrm{SR}=0.3$, but detail loss, structural
distortions, and color artifacts become more apparent as
the sampling ratio decreases. OE-QCP generally preserves
details, structural continuity, and colors more effectively,
with a clearer visual advantage at $\mathrm{SR}=0.1$.
These observations agree with the quantitative comparisons.

\subsection{Ablation Studies}
To assess the effectiveness of the order-ensemble mechanism,
we compare OE-QCP with OS-QCP under the three representative
orders \(123\), \(132\), and \(213\).
The experiments use the same eight video sequences and
settings as in Section~6.2, with
\(\mathrm{SR}\in\{0.1,0.2,0.3\}\).
Figs.~\ref{fig:PSNR_comparison} and~\ref{fig:ssim_comparison}
report the PSNR and SSIM results, respectively.

\begin{figure}[htbp]
	\centering
	\begin{subfigure}{0.90\linewidth}
		\centering
		\includegraphics[
		width=\linewidth,
		height=0.24\textheight
		]{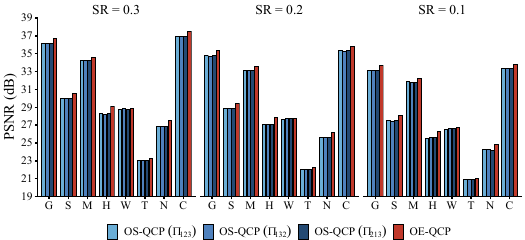}
		\caption{PSNR}
		\label{fig:PSNR_comparison}
	\end{subfigure}
	
	\smallskip
	
	\begin{subfigure}{0.90\linewidth}
		\centering
		\includegraphics[
		width=\linewidth,
		height=0.24\textheight
		]{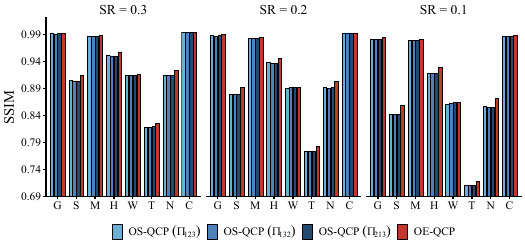}
		\caption{SSIM}
		\label{fig:ssim_comparison}
	\end{subfigure}
	\caption{Comparison of the three fixed-order OS-QCP
		variants and OE-QCP on eight video sequences at
		\(\mathrm{SR}=0.3,0.2,0.1\).
		The horizontal-axis labels abbreviate the video names:
		G (Grandma), S (Silent), M (Mother and Daughter),
		H (Hall Monitor), W (Waterfall), Su (Suzie),
		N (News), and C (Claire).}
	\label{fig:ablation_quantitative}
\end{figure}

Fig.~\ref{fig:ablation_quantitative} 
shows that the three fixed-order OS-QCP variants achieve similar
PSNR and SSIM values, with no single order consistently
outperforming the others. OE-QCP achieves better results,
with larger gains generally observed at lower sampling ratios.
These results demonstrate the benefit of incorporating multiple
multiplication orders, particularly under sparse observations.

\begin{figure}[htbp]
	\centering
	\setlength{\tabcolsep}{0pt}
	\begin{tabular}{@{}ccccc@{}}
		Original &
		OS-QCP (\(123\)) &
		OS-QCP (\(132\)) &
		OS-QCP (\(213\)) &
		OE-QCP \\[3pt]
		
		\includegraphics[width=0.20\linewidth]
		{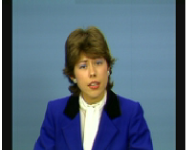} &
		\includegraphics[width=0.20\linewidth]
		{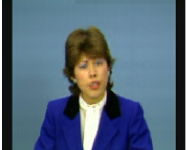} &
		\includegraphics[width=0.20\linewidth]
		{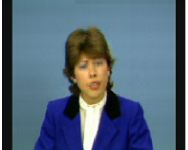} &
		\includegraphics[width=0.20\linewidth]
		{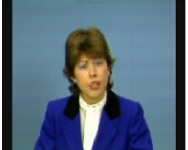} &
		\includegraphics[width=0.20\linewidth]
		{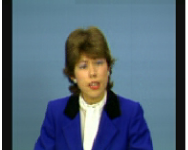} \\[5pt]
		
		\includegraphics[width=0.20\linewidth]
		{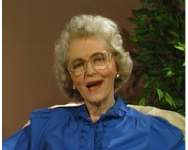} &
		\includegraphics[width=0.20\linewidth]
		{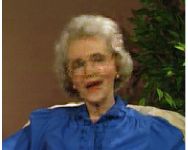} &
		\includegraphics[width=0.20\linewidth]
		{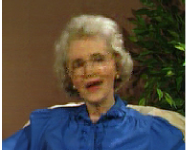} &
		\includegraphics[width=0.20\linewidth]
		{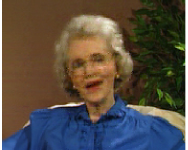} &
		\includegraphics[width=0.20\linewidth]
		{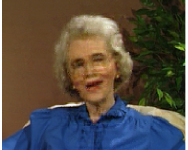} \\[5pt]
		
		\includegraphics[width=0.20\linewidth]
		{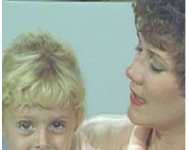} &
		\includegraphics[width=0.20\linewidth]
		{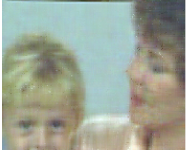} &
		\includegraphics[width=0.20\linewidth]
		{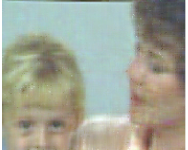} &
		\includegraphics[width=0.20\linewidth]
		{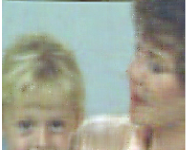} &
		\includegraphics[width=0.20\linewidth]
		{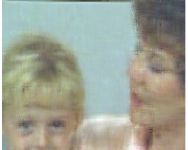}
	\end{tabular}
	
	\caption{Visual comparison of OS-QCP and OE-QCP.
		From top to bottom: \textit{Claire} at
		\(\mathrm{SR}=0.3\), \textit{Grandma} at
		\(\mathrm{SR}=0.2\), and \textit{Mother and Daughter}
		at \(\mathrm{SR}=0.1\).}
	\label{fig:ablation_visual}
\end{figure}

The visual comparisons in
Fig.~\ref{fig:ablation_visual}
agree with the quantitative results. The three fixed-order
OS-QCP variants yield similar overall quality but differ in
local textures, boundaries, and fine details. OE-QCP generally
reduces reconstruction artifacts and preserves local structures
more effectively.

\section{Conclusion}
\label{sec:conclusion}

This paper developed an order-specified quaternion CP (OS-QCP)
decomposition with quaternion-valued factor matrices and an
explicitly prescribed multiplication order. The theoretical
analysis established that multiplication order affects the
decomposition itself, even when the factors are freely chosen:
the same tensor can have different OS-QCP ranks under different
orders, and equal ranks do not guarantee equivalence of the
factors under quaternion scaling and column permutations.
We characterized these indeterminacies and established a
sufficient condition for essential uniqueness for any prescribed
multiplication order. An explicit example further showed that
an essentially unique CP decomposition of the complex
representation does not eliminate OS-QCP order dependence.
These results establish multiplication order as a structural
part of quaternion CP decomposition.

For computation, an ALS algorithm was derived for the canonical
order by exploiting the real-linear operators induced by left
and right quaternion multiplication, reducing the factor updates
to real least-squares problems. Building on OS-QCP, we introduced
an order-ensemble quaternion CP (OE-QCP) model that combines
the six multiplication orders through nonnegative weights
using a single set of three quaternion factor matrices.
The model includes every single-order OS-QCP model as a
special case and was applied to quaternion tensor completion.
A parallel row-wise masked ALS algorithm was developed,
with a convergence guarantee under the stated assumptions.
Experiments on synthetic tensors demonstrated order sensitivity
and the ability of OE-QCP to adapt to the generating order,
while color video approximation and completion experiments
demonstrated the practical benefit of incorporating multiple
orders.

The theoretical and numerical results together show that
quaternion noncommutativity has consequences beyond changing
products of fixed factors: it affects the component counts
and factor structures needed to represent a tensor.
OE-QCP incorporates this order dependence into a model in
which the factors and order weights are jointly estimated.
Further work will investigate essential uniqueness for
OE-QCP, more detailed relationships between decompositions
under different orders, and efficient algorithms for
large-scale problems. Extending the order-specified principle
to other quaternion tensor decompositions is another direction
for future research.

\bigskip
\noindent\textbf{Data Availability} Enquiries about data availability should be directed to the authors.
\section*{Declarations}
\noindent\textbf{Conflict of interest} The authors declare that they have no conflict of interest.

\bibliographystyle{spmpsci}      % mathematics and physical sciences
\bibliography{references}

@article{sidiropoulos2017tensor,
	title={Tensor decomposition for signal processing and machine learning},
	author={Sidiropoulos, Nicholas D and De Lathauwer, Lieven and Fu, Xiao and Huang, Kejun and Papalexakis, Evangelos E and Faloutsos, Christos},
	journal={IEEE Transactions on Signal Processing},
	volume={65},
	number={13},
	pages={3551--3582},
	year={2017},
	publisher={IEEE}
}

@article{cichocki2015tensor,
	title={Tensor decompositions for signal processing applications: From two-way to multiway component analysis},
	author={Cichocki, Andrzej and Mandic, Danilo and De Lathauwer, Lieven and Zhou, Guoxu and Zhao, Qibin and Caiafa, Cesar and Phan, Huy Anh},
	journal={IEEE Signal Processing Magazine},
	volume={32},
	number={2},
	pages={145--163},
	year={2015},
	publisher={IEEE}
}

@article{hong2020generalized,
	title={Generalized canonical polyadic tensor decomposition},
	author={Hong, David and Kolda, Tamara G and Duersch, Jed A},
	journal={SIAM Review},
	volume={62},
	number={1},
	pages={133--163},
	year={2020},
	publisher={SIAM}
}

@article{harshman1970foundations,
	title={Foundations of the PARAFAC procedure: Models and conditions for an “explanatory” multi-modal factor analysis},
	author={Harshman, Richard A and others},
	journal={UCLA Working Papers in Phonetics},
	volume={16},
	number={1},
	pages={84},
	year={1970},
	publisher={Los Angeles, CA}
}

@article{kolda2009tensor,
	title={Tensor decompositions and applications},
	author={Kolda, Tamara G and Bader, Brett W},
	journal={SIAM Review},
	volume={51},
	number={3},
	pages={455--500},
	year={2009},
	publisher={SIAM}
}

@article{liu2019low,
	title={Low CP rank and tucker rank tensor completion for estimating missing components in image data},
	author={Liu, Yipeng and Long, Zhen and Huang, Huyan and Zhu, Ce},
	journal={IEEE Transactions on Circuits and Systems for Video Technology},
	volume={30},
	number={4},
	pages={944--954},
	year={2019},
	publisher={IEEE}
}

@article{chen2022unsupervised,
	title={Unsupervised feature selection via graph regularized nonnegative CP decomposition},
	author={Chen, Bilian and Guan, Jiewen and Li, Zhening},
	journal={IEEE Transactions on Pattern Analysis and Machine Intelligence},
	volume={45},
	number={2},
	pages={2582--2594},
	year={2022},
	publisher={IEEE}
}

@article{chabriel2014joint,
	title={Joint matrices decompositions and blind source separation: A survey of methods, identification, and applications},
	author={Chabriel, Gilles and Kleinsteuber, Martin and Moreau, Eric and Shen, Hao and Tichavsky, Petr and Yeredor, Arie},
	journal={IEEE Signal Processing Magazine},
	volume={31},
	number={3},
	pages={34--43},
	year={2014},
	publisher={IEEE}
}

@article{zhao2022fast,
	title={Fast weighted CP decomposition for context-aware recommendation with explicit and implicit feedback},
	author={Zhao, Jianli and Zheng, Shidong and Huo, Huan and Gong, Maoguo and Zhang, Tianheng and Qu, Lijun},
	journal={Expert Systems with Applications},
	volume={204},
	pages={117591},
	year={2022},
	publisher={Elsevier}
}

@article{miao2021color,
	title={Color image recovery using low-rank quaternion matrix completion algorithm},
	author={Miao, Jifei and Kou, Kit Ian},
	journal={IEEE Transactions on Image Processing},
	volume={31},
	pages={190--201},
	year={2021},
	publisher={IEEE}
}

@article{jia2025new,
	title={A new cross-space total variation regularization model for color image restoration with quaternion blur operator},
	author={Jia, Zhigang and Xiang, Yuelian and Zhao, Meixiang and Wu, Tingting and Ng, Michael K},
	journal={IEEE Transactions on Image Processing},
	volume={34},
	pages={995--1008},
	year={2025},
	publisher={IEEE}
}

@article{chen2022color,
	title={Color image inpainting via robust pure quaternion matrix completion: Error bound and weighted loss},
	author={Chen, Junren and Ng, Michael K},
	journal={SIAM Journal on Imaging Sciences},
	volume={15},
	number={3},
	pages={1469--1498},
	year={2022},
	publisher={SIAM}
}

@article{he2023eigenvalues,
	title={Eigenvalues of quaternion tensors with applications to color video processing},
	author={He, Zhuo-Heng and Wang, Xiang-Xiang and Zhao, Yun-Fan},
	journal={Journal of Scientific Computing},
	volume={94},
	number={1},
	pages={1},
	year={2023},
	publisher={Springer}
}

@article{jia2022non,
	title={Non-local robust quaternion matrix completion for large-scale color image and video inpainting},
	author={Jia, Zhigang and Jin, Qiyu and Ng, Michael K and Zhao, Xi-Le},
	journal={IEEE Transactions on Image Processing},
	volume={31},
	pages={3868--3883},
	year={2022},
	publisher={IEEE}
}

@article{kawai2025quaternion,
	title={Quaternion reservoir computing for spatiotemporal analysis in polarimetric synthetic aperture radar},
	author={Kawai, Kitoshi and Konishi, Bungo and Natsuaki, Ryo and Hirose, Akira},
	journal={Neurocomputing},
	pages={131633},
	year={2025},
	publisher={Elsevier}
}

@article{miao2020quaternion,
	title={Quaternion-based bilinear factor matrix norm minimization for color image inpainting},
	author={Miao, Jifei and Kou, Kit Ian},
	journal={IEEE Transactions on Signal Processing},
	volume={68},
	pages={5617--5631},
	year={2020},
	publisher={IEEE}
}

@article{song2021low,
	title={Low rank pure quaternion approximation for pure quaternion matrices},
	author={Song, Guangjing and Ding, Weiyang and Ng, Michael K},
	journal={SIAM Journal on Matrix Analysis and Applications},
	volume={42},
	number={1},
	pages={58--82},
	year={2021},
	publisher={SIAM}
}

@article{chen2019low,
	title={Low-rank quaternion approximation for color image processing},
	author={Chen, Yongyong and Xiao, Xiaolin and Zhou, Yicong},
	journal={IEEE Transactions on Image Processing},
	volume={29},
	pages={1426--1439},
	year={2019},
	publisher={IEEE}
}

@article{miao2020low,
	title={Low-rank quaternion tensor completion for recovering color videos and images},
	author={Miao, Jifei and Kou, Kit Ian and Liu, Wankai},
	journal={Pattern Recognition},
	volume={107},
	pages={107505},
	year={2020},
	publisher={Elsevier}
}

@article{yang2025randomized,
	title={Randomized quaternion tensor UTV decompositions for color image and color video processing},
	author={Yang, Liqiao and Miao, Jifei and Jiang, Tai-Xiang and Zhang, Yanlin and Kou, Kit Ian},
	journal={Pattern Recognition},
	volume={165},
	pages={111580},
	year={2025},
	publisher={Elsevier}
}

@article{yang2026nonlinear,
	title={Nonlinear Transformed Low-Rank Quaternion Tensor Total Variation for Multidimensional Color Image Completion},
	author={Yang, Liqiao and Hu, Yexun and Jiang, Tai-Xiang and Wei, Yimin and Liu, Guisong and Ng, Michael K},
	journal={IEEE Transactions on Image Processing},
	year={2026},
	publisher={IEEE}
}

@article{flamant2026multilinear,
	title={Multilinear analysis of quaternion arrays: theory and computation},
	author={Flamant, Julien and Luciani, Xavier and Miron, Sebastian and Zniyed, Yassine},
	journal={SIAM Journal on Matrix Analysis and Applications},
	volume={47},
	number={1},
	pages={158--184},
	year={2026},
	publisher={SIAM}
}

@book{hamilton1866elements,
	title={Elements of quaternions},
	author={Hamilton, William Rowan},
	year={1866},
	publisher={London: Longmans, Green, \& Company}
}

@article{de2008tensor,
	title={Tensor rank and the ill-posedness of the best low-rank approximation problem},
	author={De Silva, Vin and Lim, Lek-Heng},
	journal={SIAM Journal on Matrix Analysis and Applications},
	volume={30},
	number={3},
	pages={1084--1127},
	year={2008},
	publisher={SIAM}
}

@article{flamant2020quaternion,
	title={Quaternion non-negative matrix factorization: Definition, uniqueness, and algorithm},
	author={Flamant, Julien and Miron, Sebastian and Brie, David},
	journal={IEEE Transactions on Signal Processing},
	volume={68},
	pages={1870--1883},
	year={2020},
	publisher={IEEE}
}

@article{miao2024quaternion,
	title={Quaternion tensor left ring decomposition and application for color image inpainting},
	author={Miao, Jifei and Kou, Kit Ian and Cai, Hongmin and Liu, Lizhi},
	journal={Journal of Scientific Computing},
	volume={101},
	number={1},
	pages={1},
	year={2024},
	publisher={Springer}
}

@ARTICLE{7010937,
  author={Zhao, Qibin and Zhang, Liqing and Cichocki, Andrzej},
  journal={IEEE Transactions on Pattern Analysis and Machine Intelligence}, 
  title={Bayesian CP Factorization of Incomplete Tensors with Automatic Rank Determination}, 
  year={2015},
  volume={37},
  number={9},
  pages={1751-1763}}

@article{zhang1997quaternions,
	title={Quaternions and matrices of quaternions},
	author={Zhang, Fuzhen},
	journal={Linear Algebra and its Applications},
	volume={251},
	pages={21--57},
	year={1997},
	publisher={Elsevier}
}

@book{rodman2014topics,
	title={Topics in quaternion linear algebra},
	author={Rodman, Leiba},
	year={2014},
	publisher={Princeton University Press}
}

@article{LiZZC25,
  author       = {Ben{-}Zheng Li and
                  Xi{-}Le Zhao and
                  Hao Zhang and
                  Delin Chu},
  title        = {Dictionary-Based Block Term Decomposition for Third-Order Tensors},
  journal      = {SIAM Journal on Imaging Sciences},
  volume       = {18},
  number       = {4},
  pages        = {2347--2377},
  year         = {2025}}

@article{Pan2022SeparableQM,
  title={Separable Quaternion Matrix Factorization for Polarization Images},
  author={Junjun Pan and Michael K. Ng},
  journal={SIAM Journal on Imaging Sciences},
  year={2022},
  volume={16},
  pages={1281-1307}}

@article{valverde2018dual,
 	title={Dual quaternion framework for modeling of spacecraft-mounted multibody robotic systems},
 	author={Valverde, Alfredo and Tsiotras, Panagiotis},
 	journal={Frontiers in Robotics and AI},
 	volume={5},
 	pages={128},
 	year={2018},
 	publisher={Frontiers Media SA}
}

\end{document}